\documentclass{amsproc}
\usepackage{amsmath}
\usepackage{amssymb}
\usepackage{graphicx}
\usepackage{multirow}
\usepackage{array}
\DeclareGraphicsExtensions{.png,.pdf,.eps}
\usepackage[all,2cell]{xy}
\usepackage{tikz}
\usepackage{pgf}
\usepackage[customcolors]{hf-tikz}
\usepackage{enumerate}
\usepackage{mathdots}
\usepackage{enumitem}

\usetikzlibrary{cd}
\usepackage{tikz-cd}
\usepackage{hyperref}
\usepackage{ytableau}
\usepackage{pythonhighlight}

\usepackage{pifont} 
\newtheorem{theorem}{Theorem}[section]
\newtheorem{lemma}[theorem]{Lemma}
\newtheorem{corollary}[theorem]{Corollary}
\newtheorem{proposition}[theorem]{Proposition}

\theoremstyle{definition}

\newtheorem{definition}[theorem]{Definition}
\newtheorem{notation}[theorem]{Notation}
\newtheorem{setup}[theorem]{Setup}

\newtheorem{remark}[theorem]{Remark}

\newtheorem{set}[theorem]{Setup}

\newtheorem{construction}{Construction}

\newcolumntype{C}{>{$}c<{$}}

\renewcommand{\arraystretch}{1.2}

\newcommand\xx{11} 
\newcommand\yy{18} 
\newcommand{\FO}{W(\tau)}

\title[Quotients of flag varieties]{Quotients of flag varieties and their birational geometry} 

\author[Barban]{Lorenzo Barban}
\address{Center for Complex Geometry, Institute for Basic Science (IBS), 55 Expo-ro, Yuseong-gu, Daejeon, 34126, Republic of Korea}
\email{lorenzobarban@ibs.re.kr}

\author[Occhetta]{Gianluca Occhetta}
\address{Dipartimento di Matematica, Universit\`a degli Studi di Trento, via
Sommarive 14 I-38123, Trento (TN), Italy}
\email{gianluca.occhetta@unitn.it}

\author[Sol\'a Conde]{Luis E. Sol\'a Conde}
\address{Dipartimento di Matematica, Universit\`a degli Studi di Trento, via Sommarive 14 I-38123, Trento (TN), Italy}
\email{eduardo.solaconde@unitn.it}

\thanks{The first author was supported by the Institute for Basic Science (IBS-R032-D1). Second and third author partially supported by INdAM--GNSAGA}

\usepackage[textsize=tiny]{todonotes}

\usepackage{enumitem}

\newcommand\ignore[1]{}

\DeclareMathOperator{\HH}{H}
\DeclareMathOperator{\codim}{codim}

\def\ex{\operatorname{exp}}

\def\conv{\operatorname{Conv}}

\newcommand\PP{{\mathbb{P}}}

\newcommand\QQ{{\mathbb{Q}}}
\newcommand\ZZ{{\mathbb{Z}}}

\def\C{{\mathbb C}}
\def\F{{\mathbb F}}

\def\P{{\mathbb P}}
\def\Q{{\mathbb Q}}
\def\R{{\mathbb R}}
\def\Z{{\mathbb Z}}

\def\cA{{\mathcal A}}

\def\cC{{\mathcal C}}

\def\cE{{\mathcal E}}

\def\cI{{\mathcal I}}

\def\cO{{\mathcal{O}}}

\def\cU{{\mathcal U}}

\def\cY{{\mathcal Y}}
\def\Q{{\mathbb{Q}}}

\def\fn{{\mathfrak n}}

\def\fg{{\mathfrak g}}
\def\fh{{\mathfrak h}}

\def\fb{{\mathfrak b}}

\def\fsl{{\mathfrak sl}}

\def\operatorname#1{\mathop{\rm #1}\nolimits}

\def\Proj{\operatorname{Proj}}
\def\Aut{\operatorname{Aut}}

\def\Bir{\operatorname{Bir}}
\def\Chow{\operatorname{Chow}}

\def\Exc{\operatorname{Exc}}

\def\Hom{\operatorname{Hom}}

\def\Pic{\operatorname{Pic}}
\def\Hom{\operatorname{Hom}}

\def\codim{\operatorname{codim}}

\def\det{\operatorname{det}}
\def\conj{\operatorname{conj}}

\def\NE{\operatorname{NE}}

\def\Nef{{\operatorname{Nef}}}

\def\Nu{{\operatorname{N_1}}}
\def\NU{{\operatorname{N^1}}}

\def\Eff{{\operatorname{Eff}}}
\def\Mov{{\operatorname{Mov}}}

\newcommand{\cNE}[1]{\overline{\NE}(#1)}

\def\PGL{\operatorname{PGL}}

\def\BD{\operatorname{BD}}

\def\Sp{\operatorname{Sp}}

\DeclareMathOperator{\No}{N}

\def\OP{\operatorname{OP}}

\def\GZ{\mathcal{G}\! Z}
\def\GU{\mathcal{G} U}

\def\CZ{\mathcal{C}\mkern-1mu Z}
\def\LZ{\mathcal{L} Z}
\def\CF{\mathcal{C} \mkern-1mu F}

\newcommand{\tast}{\mathbin{\mkern-2.5mu * \mkern-2.5mu}}

\newcommand{\pb}{\ar@{}[dr]|{\text{\pigpenfont J}}}
\def\ol{\overline}

\makeatletter
\newcommand{\xleftrightarrow}[2][]{\ext@arrow 3359\leftrightarrowfill@{#1}{#2}}
\makeatother
\newcommand{\xdasharrow}[2][->]{
\tikz[baseline=-\the\dimexpr\fontdimen22\textfont2\relax]{
\node[anchor=south,font=\scriptsize, inner ysep=1.5pt,outer xsep=2.2pt](x){#2};
\draw[shorten <=3.4pt,shorten >=3.4pt,dashed,#1](x.south west)--(x.south east);
}}

\newcommand\lra{\longrightarrow}

\def\Mo{\operatorname{\hspace{0cm}M}}

\begin{document}
\begin{abstract}
We compute the Chow quotient of the complete flag variety of subspaces of a four dimensional complex vector space, show that it is smooth and a Mori Dream Space, and describe in detail its birational geometry. 
\end{abstract}
\maketitle

\tableofcontents

\section{Introduction}\label{sec:intro}

It is a widespread idea in mathematics that one can use group actions to simplify problems, from the resolution of differential equations to geometry and mathematical physics. In algebraic geometry, many interesting objects -- such as moduli spaces -- can be constructed as quotients of varieties by algebraic group actions. In this context, the foundational work of Mumford in the 1960s was crucial in formalizing the notion of algebro-geometric quotients of varieties (cf. \cite{MFK}). 

One distinctive feature of GIT quotients is that they depend on a choice of linearization of the action on an ample line bundle over the variety. A natural question, then, is how the resulting quotients vary with the choice of the  linearization. This question was later explored within the framework of birational geometry, leading to several interesting concepts and results (cf. \cite{DolgachevHu}, \cite{HuKeel}).

On the other hand, Kapranov, Bia{\l}ynicki-Birula, Sommese, et al. (see \cite{Kapranov}, \cite{BBS3}) considered the problem of constructing an intrinsic notion of quotients by group actions, and proposed the closely related concepts of Chow and Hilbert quotients of a variety.

Moreover, the landmark paper \cite{KSZ} by Kapranov, Sturmfels, and Zelevinsky showed that Chow quotients can be computed combinatorially in the case of projective toric varieties with the action of a subtorus. This work also led to the definition of combinatorial quotient for toric varieties, which can be defined even in the non-projective case.

A remarkable point  is that Chow quotients by the action of complex tori of complex rational homogeneous varieties -- which are among the simplest and most symmetric algebraic varieties -- provide extremely interesting examples of varieties. 
For instance, Kapranov showed that the Grothendieck--Knudsen moduli space of stable genus zero curves with marked points can be realized as the Chow quotient of a Grassmannian by the action of a maximal torus in its automorphism group (cf. \cite[Theorem~4.1.8]{Kapranov}). This stimulated other researchers, such as Thaddeus, who characterized certain classical moduli spaces -- including the space of collineations -- as Chow quotients (see \cite[Main~Theorem]{Thaddeus}).

A natural problem  -- that we address in this work -- is to extend Kapranov's study to Chow quotients of other rational homogeneous varieties by the action of a maximal torus.

Rational homogeneous varieties are defined as quotients of semisimple algebraic groups by parabolic subgroups, and occur in families, classified by the underlying semisimple group $G$.
For a fixed $G$, all rational homogeneous $G$-varieties arise as smooth fiber type contractions of the associated complete flag manifold, which is the quotient of $G$ by a Borel subgroup $B$. 
This is the rational homogeneous variety that is ``closest'' to $G$, and in fact it has been shown that the homogeneous structure of $G/B$ can be reconstructed out of geometric properties of 
its elementary contractions (cf. \cite{OSWW}, \cite[Appendix A]{OSWi}). It makes then sense to start the research by looking at Chow quotients of complete flag varieties, from which the rest of quotients of rational homogeneous varieties can be obtained by means of contractions (see Section \ref{ssec:Winvcontr}). 

The geometry of these Chow quotients can be very involved, and substantially different from the varieties defining them. First of all, they can be singular, as one can check for the complete flag associated to the group $\Sp(4)$ (cf. \cite{Bianco}). Then, while the geometry of a complete flag $F=G/B$ is determined by the action of the group $G$, the Chow quotient $X$ of $F$ is naturally endowed with the action of the Weyl group $W$ of $G$. Moreover, while all the contractions of $F$ are  smooth fibrations, its quotient $X$ enjoys a frantic birational life.

One can easily show that the Chow quotient of the complete flag manifold associated to  $\PGL(3)$ is $\P^1$, hence the first interesting case is the complete flag variety of vector subspaces of $\C^4$. Our main result is the following:

\begin{theorem}\label{thm:main}
Let $F$ be the complete flag variety of vector subspaces of $\C^4$, endowed with the natural action of the complex torus $H\subset\PGL(4)$ of homothety classes of diagonal matrices. Then the Chow quotient $X$ of $F$ by the action of $H$ is a smooth rational weak Fano threefold of Picard number $12$. 
\end{theorem}

A large part of the paper is then devoted to the description of the birational properties of $X$. 
We will show that $X$ admits the action of an index two extension of the Weyl group, namely, the  octahedral symmetry group $S_4\rtimes\Z/2\Z$. This action is then inherited by the space of numerical class of divisors $\NU(X)$, and by the nef and effective cones of $X$, which we will describe. 

A key point in our approach is the fact that $F$ admits an $H$-invariant open covering by affine spaces, which are $H$-equivariantly isomorphic to the Lie algebra $\fn$ of nilpotent lower triangular $4\times 4$ matrices. Roughly speaking, the combinatorial quotient of $\fn$ by the action of $H$, that can be computed in a reasonably easy way, helps us understanding the possible degenerations of the general $H$-invariant cycle parametrized by $X$. We show that $X$ admits $24$ (one for every fixed point of the action of $H$ on $F$) birational contractions onto the combinatorial quotient of $\fn$, which is a smooth toric projective variety. By analyzing these contractions, we are able to reconstruct $X$.

We remark that in \cite[Theorem~1.3]{CoreyOlarte}  it was already shown, using other methods, that the Chow quotient 
$X$ is smooth (along with being a log crepant resolution of the log canonical compactification of an open subset of a subvariety of $F$ parametrizing flags with a fixed base point). 
Our approach is different, and we give an explicit description of $X$.\par\medskip

\noindent{\bf Outline.}
In Section \ref{sec:prelim} we recall background results on torus actions and flag varieties, and introduce the terminology that will be used throughout the paper. In Section \ref{sec:Chow} we introduce the Chow quotient $X$ of $F$ and the closure of the $H$-orbit of a  general point of $F$, which is a projective toric variety whose moment polytope is a permutohedron. We also introduce the notion of fundamental subtori of $H$. These are closely related to the divisors of 
$X$ that parametrize codimension-one degenerations of the closure of the general orbit; we call them boundary divisors.

The combinatorial quotients $X_\sigma$ of $H$-invariant affine charts in $F$ are computed in Section \ref{sec:local}. These are projective toric varieties, obtained from $\P^3$ via a sequence of toric blowups. We show that some of their torus-invariant divisors are related to the fundamental subtori introduced earlier. 
We then show that the combinatorial quotients  $X_\sigma$ arise as inverse limits of certain GIT quotients of 
$F$, from which it follows that  each $X_\sigma$ is the target of a contraction of $X$.

The next step, addressed in Section \ref{sec:merging}, is the reconstruction of the Chow quotient
 from the combinatorial quotients $X_\sigma$. We show that the natural birational maps among the $X_\sigma$'s are extensions of the elements of a finite subgroup $\FO$ of the Cremona group of $\P^3$, that we call {\em tile group}, and 
describe in full detail. With these ingredients at hand, we show in Section \ref{sec:constructX} that $X$ can be identified with a variety, which we call the {\em tile threefold}, obtained from $\P^3$ via a sequence of blowups that resolves the indeterminacies of the elements of $\FO$. A detailed description of this construction shows, in particular, that 
$X$ is a smooth Fano threefold of Picard number $12$ (Proposition \ref{prop:XisoX'}).

The last two sections of the paper focus on the birational geometry of the variety $X$. After studying its intersection theory in terms of boundary divisors -- which generate $\Pic(X)$ -- we prove that $X$ is a smooth weak Fano manifold, whose anticanonical model has degree $12$ in $\P^{13}$ (cf. Section \ref{ssec:antican}); in particular this concludes the proof of Theorem \ref{thm:main}. 

As a consequence $X$ is a Mori Dream Space, therefore its Mori, Nef and Effective cones are polyhedral. We compute these cones in Section \ref{sec:birgeomX}, and study some of the contractions of $X$. In particular, we identify the Chow quotients of partial flag varieties of $\C^4$ as targets of contractions of $X$, which we describe explicitly.

Some of the results in this paper rely on explicit computations, carried out with the aid of the software {\tt SageMath}. For the reader's convenience, our scripts are available at \cite{AzulSage}, organized in two files: {\tt Chow\_aux\_1.ipynb}, corresponding to Sections \ref{sec:local} and \ref{sec:merging}, and {\tt Chow\_aux\_2.ipynb}, corresponding to Section \ref{sec:birgeomX}.\par \medskip 

\noindent{\bf Acknowledgements.} We would like to thank J. Hong for presenting this problem to us, as well as J. Wi\'sniewski, M. Donten-Bury and E. Lee for interesting conversations and helpful references. 

\section{Preliminaries}\label{sec:prelim}

\subsection{Notation}\label{ssec:notn}

We work over the field of complex numbers.  Given a free abelian group $\Mo$, we denote by $\Mo_{\R}$ the associated real vector space. 

Let $Z$ be a normal projective variety. We denote by $\rho_Z$ the \emph{Picard number} of $Z$, by $\NU(Z)$ the real vector space of Cartier divisors modulo numerical equivalence, by $\Nu(Z)=\NU(Z)^\vee$ the space of numerical classes of $1$-cycles, and by $\Nef(Z)\subset \NU(Z)$, $\cNE{Z}\subset \Nu(Z)$ the {\em nef} and {\em Mori cone of $Z$}, respectively.

A \emph{contraction} is a surjective morphism $\phi: Z\to Y$ with connected fibers onto a normal projective variety. 
The contraction $\phi$ is {\em elementary} if $\rho_Z-\rho_Y=1$. If $\dim Z > \dim Y$, the contraction is \emph{of fiber type}, otherwise it is birational. If $\phi$ is birational and $\codim \Exc(\phi) \geq 2$, $\phi$ is called \emph{small}, otherwise it is called \emph{divisorial}.
Given a Mori extremal ray $R$ of $Z$, we denote by $\varphi_R$ the associated contraction, by $\Exc(R)$ its exceptional locus, by  $\ell(R)$ its {\em length}, i.e., the minimum anticanonical degree of curves contracted by $\varphi_R$. 

A birational map $\phi:Z\dashrightarrow Y$ is a \emph{small modification} if it is an isomorphism in codimension $1$. If $Z$ and $Y$ are $\Q$-factorial, such a map is called a \emph{small $\Q$-factorial modification} (SQM, for short). Let $\widehat{Z}$ be a small $\Q$-factorial modification of $Z$. We will denote by $\widehat{E}$ the strict transform in $\widehat{Z}$ of a divisor $E \subset Z$.

A normal $\Q$-factorial projective variety $Z$ is called a {\it Mori Dream Space} (MDS for short) if 
\begin{enumerate}[leftmargin=\yy pt]
	\item\label{item:1} the Picard group $\Pic(Z)$ of $Z$ is finitely generated;
	\item\label{item:2} the nef cone $\Nef(Z)$ is generated by classes of finitely many semiample divisors;
	\item\label{item:3} there exists a finite number of SQMs $f_i:Z\dashrightarrow Z_i$, for $i=0,\ldots,k$, such that every $Z_i$ satisfies $(\ref{item:2})$ and $$\Mov(Z)=\bigcup_{i=0}^k f_i^*\Nef(Z_i).$$
\end{enumerate}

\subsection{Torus actions and associated quotients}\label{ssec:prelimtorus}

In this section we briefly introduce the necessary background on algebraic torus actions on polarized varieties and their quotients. We refer to \cite{BWW}, \cite{WORS2} for details. 

For simplicity, we will consider only the case of the nontrivial action of an algebraic torus of rank $r$ on a smooth projective variety $Z$. We denote by  $\Mo(H)$ the character lattice of $H$. 
The fixed point locus decomposes as $Z^H=\bigsqcup_{Y\in \cY} Y$, where each $Y$ is smooth (cf. \cite[Theorem~1.1]{IVERSEN}) hence irreducible.

Given an ample line bundle $L$ on $Z$, we call $(Z,L)$ a \emph{polarized pair}. A \emph{linearization} of $L$ by $H$ is an $H$-action $H\times L\to L$ which is linear along the fiber and such that the projection $L\to Z$ is $H$-equivariant. In this setting linearizations always exist \cite[Remark~2.4]{KKLV}, and two linearizations of $L$ differ by a character of $H$. For $p\in Y\subset Z^H$, we denote by $\mu_L(p)$ the \emph{weight} of the $H$-action on $L_p$.  If $p,q$ belong to the same irreducible component $Y\subset Z^H$, we have that $\mu_L(p)=\mu_L(q)$, hence we define $\mu_L(Y):=\mu_L(p)$, $p\in Y$. 
We denote by $\Delta$ the {\em polytope of fixed points of the action}, that is the convex hull in $\Mo(H)_{\R}$ of $\mu_L(Y)$, for $Y\in \cY$.

The $H$-action on the pair $(Z,L)$ induces an $H$-decomposition 
\[\HH^0(Z,L)=\bigoplus_{\tau\in \widetilde{\Gamma}}\HH^0(Z,L)_{\tau},
\] 
where $\HH^0(Z,L)_{\tau}$ denotes the eigenspace on which $H$ acts with weight $\tau$ and $\widetilde{\Gamma}$ is the set of weights of $H$-action on $\HH^0(Z,L)$. The \emph{polytope of sections} $\Gamma$ is the convex hull in $\Mo(H)_{\R}$ of $\widetilde{\Gamma}$. If $L$ is base point free, then $\Gamma = \Delta$ (cf. \cite[Lemma~2.4~(3)]{BWW}). For any $u\in \Gamma$, we set
$$\cA_u := \bigoplus_{m\geq 0, mu\in \Z^r}\HH^0(Z,mL)_{m u}.$$
The $\C$-algebra $\cA_{u}$ is finitely generated, and the normal projective variety 
$$\GZ_u:= \Proj \cA_{u}.$$
is a {\em GIT-quotient of $(Z,L)$} under the $H$-action. It parametrizes the orbits of a semistable open subset $Z^{ss}_u\subset Z$, defined in terms of the linearization of $L$ given by the weight $u\in \Gamma$. 

\subsubsection{Chow quotients.} 
We recall here the definition of Chow quotient given in \cite[Section 0.1]{Kapranov}. Let an algebraic torus $H$  act on a smooth projective variety $Z$. We may find an open subset $U$ of $Z$ of points $z$ such that the cycles $\overline{H\cdot z}$ have the same dimension and belong to the same homology class; after shrinking we may assume that $U$ is $H$-invariant and that there exists a geometric quotient $\mathcal{G}U$ of $U$ by $H$. We then have an embedding of $\mathcal{G}U$ into the {\em Chow variety of $Z$}, $\psi: \mathcal{G}U\hookrightarrow \Chow(Z)$. 

\begin{definition}\cite{Kapranov}\label{def:ChowQuotient}
	The {\em Chow quotient of $Z$  by the action of} $H$ is defined as the closure of $\psi(\GU)$, that is	
	\[\overline{\CZ}:=\overline{\psi(\mathcal{G}U)}\subset \Chow(Z).\]
	We denote by $\CZ$ the normalization of $\overline{\CZ}$, and call it the \emph{normalized Chow quotient} of $Z$ by $H$. 
	The pullback of the universal family of $\Chow(Z)$ to $\CZ$ will be denoted by $p: \cU\to \CZ$, and the evaluation morphism by $q:\cU\to Z$. An important property of this quotient is that it dominates all the GIT quotients of $Z$:
\end{definition}

\begin{theorem}\cite[Theorem~(0.4.3)]{Kapranov}\label{thm:invlim}
	For each $v\in \Gamma$ there exists a birational morphism $\pi_v:\CZ\to \GZ_v$. 
\end{theorem}

\subsubsection{Limit quotients.}  Consider again the above setting. 
Following \cite[\S 2]{BakerChow}, we can describe a stronger relation among $\CZ$ and the GIT quotients of $Z$. 
Let $u,v\in \Gamma$ be such that $Z^{ss}_v\subset Z^{ss}_u$.  Taking quotients we obtain a dominant projective morphism $f_{v,u}: \GZ_v \to \GZ_u$.  One may show that the collection of GIT quotients $\GZ_v$, for $v\in \Gamma$, together with the morphisms $f_{v,u}$, for $u\in \Gamma$ such that $Z^{ss}_v \subset Z^{ss}_u $, form an inverse system. The inverse limit of this system might be reducible 
(cf. \cite[p.~654]{KSZ}, \cite[Counterexample~1.11]{Thaddeus1996}), but there exists a unique irreducible component  $\overline{\LZ}$ containing the geometric quotient $\GU$. We call $\overline{\LZ}$ the {\em limit quotient of $Z$}, 
and its normalization $\LZ$ the {\em normalized limit quotient of $Z$} by the action of $H$. 

By definition, the normalized limit quotient comes with birational morphisms $\chi_v: \LZ\to \GZ_v$, for $v\in \Gamma$. Moreover, it enjoys the following universal property: 

\begin{remark}\label{rem:UniversalPropertyLimit}
	Given an irreducible variety $M$ and a collection of birational morphisms $\phi_v:M\to \GZ_v$ such that $\phi_u = f_{v,u}\circ \phi_v$, for $u,v\in \Gamma$ such that $X^{ss}_v\subset X^{ss}_u$, there exists a unique $\Phi : M\to \LZ$ such that $\phi_v = \chi_v \circ \Phi$ for every $v\in \Gamma$.
\end{remark}

In our setting --more generally for torus actions on normal projective varieties, cf \cite[Corollary~2.6]{BakerChow}--, 
the normalized limit quotient is equal to the normalized Chow quotient:

\begin{proposition}
\label{prop:BHR}
	Let $H$ be an algebraic torus acting nontrivially on a smooth projective variety $Z$. Then
	$\CZ\simeq \LZ$.
\end{proposition}

\subsubsection{Combinatorial quotients.}
We complete this section by recalling the construction of  
the {\em combinatorial quotients} of toric varieties, introduced in \cite{KSZ}. 
Let $Z$ be a normal toric variety, not necessarily projective, with fan $\Sigma$, let $T$ be the maximal torus acting on $Z$, and let $H\subset T$ be a subtorus. Call $\pi: \No(T)_{\R}\to \No(T/H)_{\R}$ the induced projection between the spaces associated to the lattices of $1$-parametric subgroups. The {\em quotient fan  of $\Sigma$ by $\pi$} is the fan $\pi(\Sigma)$ in $\No(T/H)_{\R}$ defined by the property that the minimal cone of the fan containing $x\in \No(T/H)_{\R}$ is 
\[
\bigcap_{\sigma\in \Sigma, x\in \pi(\sigma)} \pi(\sigma).\]
The normal toric $T/H$-variety whose fan is $\pi(\Sigma)$ is called the {\em combinatorial quotient of $Z$ by the action of $H$}. 

\begin{proposition}\cite[Theorem~2.1~(a)]{KSZ}\label{prop:ChowToric}
Let $Z$ be a projective toric $T$-variety, and $H\subset T$ be a subtorus. Then the Chow quotient of $Z$ by the action of $H$ is isomorphic to the combinatorial quotient of $Z$ by the action of $H$. 
\end{proposition}

Note that the combinatorial quotient is defined even if the toric variety $Z$ is not projective. In this paper we will compute the Chow quotient of a projective non toric variety, and we will construct it upon the combinatorial quotients of some toric, non projective, invariant open subsets. 

The quotient fan $\pi(\Sigma)$ is not necessarily equal to the union of all the projected cones. Indeed, it may happen that the projections of two cones $\sigma,\sigma'$ intersect in a subset which is not a face of $\pi(\sigma)$ or $\pi(\sigma')$. The following definition will be used in Section \ref{ssec:boundarydiv}:

\begin{definition}\label{def:rel}
	A cone $\sigma\in\Sigma$ is \emph{relevant} for $\pi$ if there exists $\sigma'\in\Sigma$ such that $\pi(\sigma)\cap \pi(\sigma') \not\preccurlyeq \pi(\sigma)$. We call $\sigma'$ a \emph{relevant companion} for $\sigma$. 
\end{definition}

\subsection{The projective linear group}\label{ssec:prelimPGL}

We recall here some known facts about the projective linear group, and introduce some notation we will use later on: 

\begin{itemize}[leftmargin= \xx pt,itemsep=4pt,after=\vspace{10pt}]
	
	\item $G=\PGL(n+1)$ will denote the {\em projective linear group}, i.e., the group of automorphisms of $\P^n:=\P(\C^{n+1})$, i.e., the homothety classes of non-singular $(n+1)\times (n+1)$ matrices. 
	
	\item $\fg=\fsl_{n+1}$ will denote the {\em Lie algebra of} $G$, that can be described as the Lie algebra of $(n+1)\times (n+1)$ matrices of trace $0$.

	\item $H\subset B\subset G$ will denote, respectively, the {\em (Cartan) subgroup} of classes of diagonal matrices, and the {\em (Borel) subgroup} of classes of upper triangular matrices; their associated Lie algebras will be denoted by $\fh\subset\fb\subset \fg$.

	\item The group of characters  $\Mo(H)$ of $H$ will be identified with the subgroup of $\Z^{n+1}$ of elements whose sum of coordinates is equal to zero. It is generated by: 
	$$
	\Delta:=\{\alpha_{i}:=e_{i-1}-e_{i}, \quad i=1,\dots,n\},
	$$
	where $\{e_0,\dots,e_n\}$ denotes the canonical $\Z$-basis of $\Z^{n+1}$; the element $\alpha_i$ corresponds to the linear map $H\to \C^*$ sending the class $\lambda$ of a diagonal matrix with diagonal entries $(\lambda_0,\lambda_1,\dots,\lambda_n)$ to $\lambda_{i-1}\lambda_i^{-1}$, $i=1,\dots,n$.
	The vector space $\Mo(H)_\R$, endowed with the scalar product $(\bullet,\bullet)$ induced by the Killing form, is then an Euclidean vector subspace of $\Z^{n+1}\otimes_\Z \R\simeq\R^{n+1}$. 
	
	\item The choice of $H\subset B\subset G$ determines a {\em root system} $\Phi\subset \Mo(H)$, and a subset of positive simple roots, which is precisely $\Delta=\{\alpha_i,\,\,i=1,\dots,n\}\subset \Phi$. The  positive elements  $\Phi^+\subset \Phi$ can be written as:
	$$\alpha_{i,j}=e_{i-1}-e_{j}, \quad i\leq j, \mbox{ and}\quad\alpha_{i,j}(\lambda)=\lambda_{i-1}\lambda_j^{-1}.$$
	Note that $\alpha_{i,i}=\alpha_{i}$, for every $i\geq 1$, $\alpha_{i,j}=\alpha_{i}+\dots +\alpha_{j}$ if $i< j$,
	and that every $\alpha_{i,j}$ has length equal $\sqrt{2}$. 
	Moreover, for every $\rho=\sum_{i=1}^nm_i\alpha_i\in \Mo(H)$, and every $\lambda=[(\lambda_0,\lambda_1,\dots,\lambda_n)]\in H$, we have:
	$$
	\lambda^{\rho}:=\rho(\lambda)=\lambda_0^{m_0}\lambda_1^{m_1-m_0}\dots\lambda_{n-1}^{m_{n-1}-m_n}\lambda_{n-1}^{m_n}.
	$$
	
\item The associated {\em Cartan decomposition of }$\fg$ will be:
		$$
		\fg=\fh\oplus \bigoplus_{i\leq j}\fg_{\alpha_{i,j}}\oplus \bigoplus_{i\leq j}\fg_{-\alpha_{i,j}},
		$$ 
		where $\fg_{\pm\alpha_{i,j}}$ denotes the $\pm\alpha_{i,j}$-eigenspace. Note that, for every $1\leq i\leq j\leq n$, $\fg_{\alpha_{i,j}}$ is generated by the nilpotent triangular matrix $U_{ij}$ whose coordinates are 
		$$
		(U_{ij})_{k\ell}=\left\{\begin{array}{ll}1&\mbox{ if }k=i-1, \ell=j,\\0&\mbox{ otherwise, }\end{array}\right.
		$$
		and $\fg_{-\alpha_{i,j}}$ is generated by $L_{ij}:=U_{ij}^t$.
		We then have that $\fb=\fh\oplus \bigoplus_{i\leq j}\fg_{\alpha_{i,j}}$; the subalgebra $\bigoplus_{i\leq j}\fg_{-\alpha_{i,j}}$, which is the Lie algebra of nilpotent lower triangular matrices, will be denoted by $\fn$.
	\item The character lattice $\Mo(H)$ is a sublattice of index $(n+1)$ of the {\em weight lattice} $\Lambda(H)\subset \Mo(H)_\R$, which is generated  the {\em fundamental dominant weights} $w_1,\dots, w_n\in \Mo(H)_\R$, defined by:
	$$2\dfrac{(w_i,\alpha_j)}{(\alpha_j,\alpha_j)}=(w_i,\alpha_j)=\delta_{ij}.$$ 
	
	\item The lattice $\No(H)$ of {\em $1$-parameter subgroups} of $H$ is the dual lattice of $\Mo(H)$. We may identify it with the weight lattice $\Lambda(H)$, by considering the map:
	$$
	\Lambda(H)\to \No(H)=\Hom(\Mo(H),\Z),\quad w\mapsto (w,\bullet).
	$$

	\item The {\em Weyl group} of $G$ associated with $H$ will be denoted by $W=\No_G(H)/H$. The reflection in $\Mo(H)_\R$ with respect to  
$\alpha_i\in \Delta$ will be denoted by $r_i\in W$. 
	
	\item $W$ is isomorphic to the group $S_{n+1}$ of permutations  of the set $\{0,1,\dots,n\}$; for every permutation $\sigma\in S_{n+1}$ we may consider the corresponding permutation matrix $P(\sigma)$, and its class in $\PGL(n+1)$, which belongs to the normalizer of $H$. It is well known that $W$ is generated  by the classes modulo $H$ of these elements. Abusing notation, we will denote by $\sigma\in\PGL(n+1)$ the class of the matrix $P(\sigma)$, for every permutation $\sigma$, and we will identify $W$ with the subgroup of $\PGL(n+1)$ formed by the classes of these elements.
	
	\item The natural action of $W$ on $H$ (conjugation) induces an action of $W$ on $\Mo(H)$ given by $\sigma(\alpha)=\alpha\circ \conj_{\sigma^{-1}}$, for every $\sigma\in W$. This action extends to a linear action on $\Mo(H)_\R$, to an action on the lattice $\Lambda(H)$, and to an action on the lattice $\No(H)$ of $1$-parametric subgroups, given by $\sigma(\mu)=\conj_{\sigma}\circ\mu$, for every $\mu\in \No(H)$, which is compatible with the isomorphism $\Lambda(H)\to \No(H)$.   
	
	\item The involutive automorphism of $\PGL(n+1)$ sending the homothety class of a matrix $A$ to the class of the conjugation  of the inverse of  $A^\top$ with the permutation $w_0:=(n,n-1,\dots,1,0)\in W$ will be called the {\em anti-transposition map} of $\PGL(n+1)$, and denoted by $\tau$. It sends $B$ to $B$ and $H$ to $H$, and  it induces an automorphism of $\Mo(H)$ that sends every $\alpha_i$ to $\alpha_{n+1-i}$. In particular, $r_i\circ \tau=\tau\circ r_{n+1-i}$, for every $i$. It is well known that the group generated by $W$ and $\tau$ is isomorphic to $W\rtimes \ZZ/2\ZZ$. In the case $n=3$ (which is the case we will be concerned with in this paper) this group is usually called the (full) {\em octahedral symmetry group}, i.e., the group of symmetries of the octahedron.
\end{itemize}		
\subsection{Fundamental subtori}\label{ssec:fundasubtori}
In this section we will define some particular subtori of $H$, that we will use later to describe some special divisors in the Chow quotient of the manifold of complete flags in $\C^4$. Let us denote by $\OP_{n+1}$ the set of {\em ordered set partitions} of $\{0,1,\dots,n\}$. An element of $\OP_{n+1}$ is a sequence $(S_0,S_1,\dots)$ of disjoint nonempty subsets with $\bigcup S_i=\{0,1,\dots,n\}$. Among them we have the trivial partition, consisting of only one subset, $\{0,1,\dots,n\}$.

\begin{definition}\label{def:fundamsubt}
Given a nontrivial ordered partition $\phi=(S_0,S_1,\dots)$ of $\{0,1,\dots,n\}$ in $(k+1)$ subsets, with $k=1,\dots,n$, we define a lattice homomorphism $\mu_{\phi}:\Mo(H)\subset \Z^{n+1}\to \Z^{k+1}$ as follows:
$$
\mu_{\phi}(m_0,m_1,\dots,m_n)=\left(\sum_{i\in S_0}m_i,\,\,\dots\,\, , \sum_{i\in S_{k}}m_i\right).
$$
Its image is the rank $k$ sublattice $\Mo(\phi)\subset\Z^{k+1}$ of elements whose sum of coordinates is equal to $0$, and which is generated by the $\Z$-basis:
$$
\{\alpha'_{i}:=e'_{i-1}-e'_{i}, \quad i=1,\dots,k\},
$$
where $\{e'_i,\,\, i=0,\dots,k\}$ denotes the canonical $\Z$-basis of $\Z^{k+1}$. 
Given a nontrivial ordered partition $\phi\in\OP_{n+1}$, the subtorus $T(\phi)$ of $H$ determined by the weight map $\mu_{\phi}$ will be called the {\em fundamental subtorus} associated to $\phi$; by definition, we have that $\Mo(T(\phi))=\Mo(\phi)$. An associated concept we will use is the following:
\end{definition}

\begin{definition}\label{def:pos1param}
Given a nontrivial partition $\phi$ as above, a {\em nonnegative $1$-pa\-ra\-me\-tric subgroup} of the subtorus $T(\phi)\subset H$ is a $1$-parametric subgroup of $T(\phi)$ determined by a nonzero group homomorphism $\rho:\Mo(\phi)\to \Z$ satisfying that $\rho(\alpha'_i)\geq 0$ for every $i=1,\dots,k$. The set of nonzero group homomorphisms $\rho$ satisfying this property will be denoted by $R(\phi)$. 
\end{definition}

Note finally that we have a natural action of $W=S_{n+1}$ on $\OP_{n+1}$, which obviously satisfies the following:
\begin{lemma}\label{lem:ordparti}
Given $\phi\in \OP_{n+1}$, and $w\in W$, we have that
$$
\mu_{w(\phi)}=\mu_{\phi}\circ w,\quad\mbox{and}\quad T(w(\phi))=\conj_w(T(\phi)).
$$
\end{lemma}

\subsection{The complete flag variety}\label{ssec:compflag}

We keep the notation of the previous sections, and recall some properties of the complete flag manifold $F:=G/B$  associated with $G$. It is well known that $F$ is a smooth projective Fano variety of dimension $\binom{n+1}{2}$ and Picard number $n$. 

There is a natural isomorphism $\Lambda(H)\simeq \Pic(F)$ defined as follows: given a weight $\lambda\in \Lambda(H)$, we have a line bundle $$L(\lambda):=G\times^B \C:=(G\times\C)/\sim,\quad (g,v)\sim (gb,\xi(b)^{\lambda}v)\mbox{ for every $b\in B$,}$$
where $\xi:B\to H$ denotes the natural map sending an upper triangular matrix to its semisimple part. Under such isomorphism the nef cone $\Nef(F)$ of $F$  can be identified with the cone generated by the fundamental dominant weights $w_i$, $i=1,\dots,n$.  
It is known that every effective divisor in $F$ is globally generated, and $\Nef(F)=\ol{\Mov}(F)=\ol{\Eff}(F)\subset \NU(F)$.  

\begin{notation}\label{not:naturalbundle}
	In the sequel we will consider the complete flag variety $F$ polarized with its {\em minimal ample line bundle}, that we define as the tensor product of the pullback of the ample generators of the Picard groups of all the contractions of $F$ to Picard number one varieties, i.e., the projective space $\P^n$, its dual and the Grassmannians of linear subspaces of $\P^n$. Notice that $L^{\otimes 2}=\cO_F(-K_F)$, and that, in weight fashion,	we may write:
$$
L=L(\lambda_{\min}),\quad\mbox{with}\quad \lambda_{\min}=\sum_{k=1}^n \dfrac{kn-k^2+k}{2}\alpha_k=\sum_{i=0}^n\left(\dfrac{n}{2}-i\right)e_i.
$$
\end{notation}


\section{Maximal torus action on the complete flag}\label{sec:Chow}

Throughout the rest of the paper we will work in the following situation:

\begin{setup}\label{setup:rest}
Let $G$ be the simple algebraic group $\PGL(4)$ of homothety classes of $4\times 4$ matrices, let $H\subset B\subset G$ be, respectively, the sets of classes of diagonal and upper triangular matrices, and  let $F=G/B$  be the complete flag variety of $\C^4$. The Weyl group of $G$ will be denoted by $W$, and we will always identify it with the group of homothety classes of permutation matrices (see Section \ref{ssec:prelimPGL}). We will consider $F$ polarized with the minimal ample line bundle $L$ (see Notation \ref{not:naturalbundle}). The normalized Chow quotient of $F$ by the action of the torus $H\subset G$ 
will be denoted by:
$$
X:=\mathcal{C}\mkern-1mu F.
$$
\end{setup}
Besides, we will freely use the notation introduced in Section \ref{sec:prelim}.

\subsection{Fixed flags and their weights}

We start by recalling the following well known statement about the $H$-action on the flag manifold $F$:

\begin{lemma}\label{lem:Bruhat}
	With the above notation: 
	\begin{enumerate}[leftmargin=\yy pt]
		\item the set of fixed points of the action of $H$ on the complete flag manifold $F$ is:
		$$
		F^H=\left\{\sigma B|\,\,\sigma\in  W\right\}.
		$$ 
		\item For every $\sigma\in  W$, the set of weights of the $H$-action on the tangent space $T_{F,\sigma B}$ is $\{\sigma (\alpha_{i,j})|\,\, i>j\}$. 
		\item For every $\lambda\in \Lambda(H)$, there exists a linearization of the $H$-action on the line bundle $L(\lambda)$ whose weights are:
		$$
		\mu^H_L(\lambda)(\sigma B)=
		-\sigma(\lambda)\in \Mo(H), \mbox{ for every }\sigma\in  W.
		$$
	\end{enumerate}
\end{lemma}

\begin{proof}
	See \cite[Propositions~3.6,~3.7]{WORS5} and the references therein.
\end{proof}

Note that, denoting by $\{\ol{e}_0,\ol{e}_1,\ol{e}_2,\ol{e}_3\}$ the canonical basis of $\C^4$, the flag corresponding to the point $\sigma B$, $\sigma\in W$, is:
$$
\{0\}\subset\langle \ol{e}_{\sigma(0)}\rangle \subset \langle \ol{e}_{\sigma(0)},\ol{e}_{\sigma(1)} \rangle \subset\langle \ol{e}_{\sigma(0)}, \ol{e}_{\sigma(1)}, \ol{e}_{\sigma(2)}\rangle \subset\C^4
$$

Note also that, for our chosen polarization, we have that:
$$
L=L(\lambda_{\min}),\quad\mbox{with}\quad \lambda_{\min}=\dfrac{1}{2}(3\alpha_1+4\alpha_2+3\alpha_3)=\sum_{i=0}^3\left(\dfrac{3}{2}-i\right)e_i,
$$
so we get
$$
\mu^H_L(\sigma B)=\sum_{i=0}^3\left(\dfrac{3}{2}-\sigma(i)\right)e_i, \mbox{ for every }\sigma\in  W.
$$
In particular, every fixed point $\sigma$ is uniquely determined by its $L$-weight, which, in coordinates with respect to $\{e_i\}$, is the permutation by $\sigma$ of the coordinates $(-3/2,-1/2,1/2,3/2)$.
Seen as points in $\Mo(H)_\R\subset\R^4$, 
these weights are the vertices of their convex hull, which is a $3$-dimensional polyhedron, classically known as the $3$-dimensional {\it permutohedron}. Throughout the paper will denote it by 

$$
\begin{array}{rl}
P:=&\conv\left(\{\mu^H_L(\sigma B),\,\,\sigma\in W\}\right)=\\[4pt]
=&\displaystyle\left\{\sum_{\sigma\in W}a_\sigma \sigma\left(-\frac{3}{2},-\frac{1}{2},\frac{1}{2},\frac{3}{2}\right),\,\, a_\sigma \in [0,1],\,\, \sum_{\sigma\in W}a_\sigma=1\right\}.
\end{array}
$$

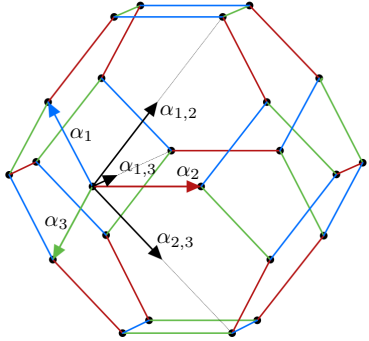
\begin{figure}[h!!]
	\centering
	\resizebox{0.44\textwidth}{!}{%
		\begin{tikzpicture}
			\node at (0, 0) {\includegraphics[width=0.5\textwidth]{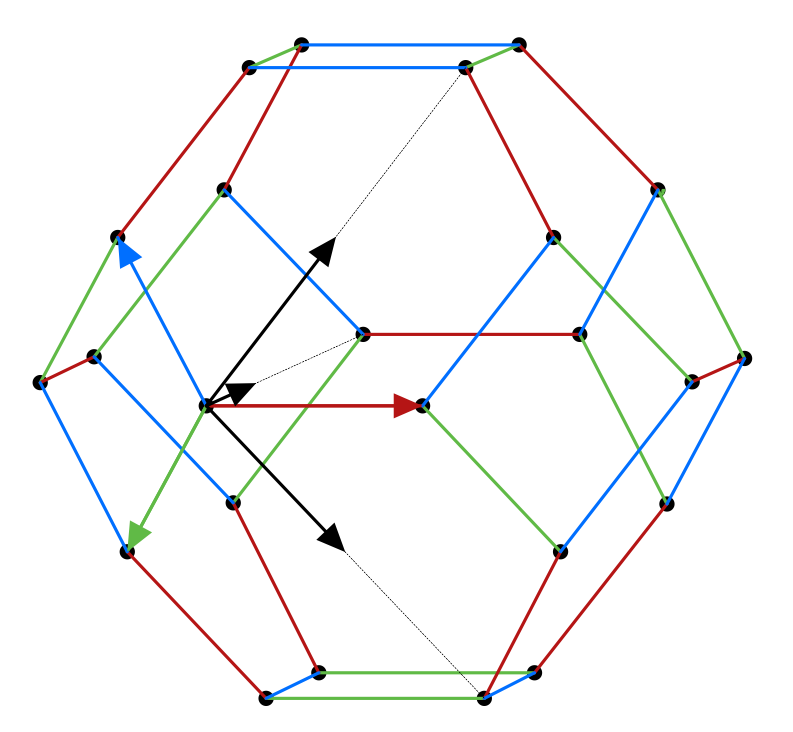}};
			\node[rotate=0] at (-1.7, 0.5) {$\alpha_1$};
			\node[rotate=0] at (-0, -0.1) {$\alpha_2$};
			\node[rotate=0] at (-2.1, -0.9) {$\alpha_3$};
			\node[rotate=0] at (-0.15, 0.85) {$\alpha_{1,2}$};
			\node[rotate=0] at (-0.2, -1.2) {$\alpha_{2,3}$};
			\node[rotate=0] at (-0.8, -0.05) {$\alpha_{1,3}$};
		\end{tikzpicture}  
	}
	\caption{The permutohedron, decorated with the basis of eigenvectors of $T_{F,eB}$.\label{fig:permuto}}
\end{figure}

\subsection{The permutohedron and the general orbit of the torus action}\label{permutagenorbit}

It is a general fact (see \cite{HuChow}, and \cite{CarrellNormal} for a proof in the homogeneous case), in the setting of torus actions on projective varieties, that the polytope generated by the weights of the action on an ample line bundle  is a moment polytope for the toric variety defined as the closure of the orbit of a general point in the variety. In our setting, this fact can be checked by means of a straightforward linear algebra argument (essentially Gauss reduction), that we include here because it will be used in the next section.  We will show the following:

\begin{proposition}\label{prop:momentgen}
In the situation of Setup \ref{setup:rest}, there exists an $H$-invariant open set $U_e\subset F$ such that $gB\in U_e$ if and only of $eB\in \ol{HgB}$.
\end{proposition}

Polarizing the toric variety $\ol{HgB}$ with the restriction of $L$, the corresponding moment polytope is the convex hull of the $L$-weights of the  fixed points of the action on $\ol{HgB}$, which  are a subset of $F^H=\{\sigma B,\,\, \sigma \in W\}$. The above statement implies  that the closure of the orbit of a point $gB$ belonging to the intersection $\bigcap_{\sigma\in W}\sigma(U_e)$ contains $F^H$:

\begin{corollary}\label{cor:momentgen}
In the situation of Setup \ref{setup:rest}, let $gB\in F$ be a general point. Then the moment polytope of the polarized projective toric variety $(\ol{HgB},L_{|\ol{HgB}})$ is equal to the permutohedron $P$.
\end{corollary}

Before proving Proposition \ref{prop:momentgen}, let us introduce the following notation:

\begin{notation}\label{notn:minors}
	Let $\sigma\in  W$ be a permutation. Given a $4\times 4$ matrix $M$, and an integer $k\in\{1,2,3\}$, we denote by $M^{\sigma}_k$ the $k\times k$ submatrix of $M$ given by
	$$
	(M^{\sigma}_k)_{i,j}=M_{\sigma(i),j}, \quad i,j\in \{0,1,\dots,k-1\}.
	$$
	Note that multiplying by a nonsingular diagonal matrix (on the left or on the right), or by a nonsingular upper triangular matrix on the right, does not change the values of $\sigma,k$ for which $\det(M^{\sigma}_k)=0$. We then define the $H$-invariant open subset:
	$$
	U_\sigma:=\left\{gB\in F|\,\, g=[M],\,\,\det(M^{\sigma}_k)\neq 0,\,\, k=1,2,3\right\}\subset F.
	$$
\end{notation}
By definition, for every $\sigma\in W$ we have:
$$
U_\sigma=\sigma(U_e).
$$
Let us now prove the above Proposition:

\begin{proof}[Proof of Proposition \ref{prop:momentgen}]
	By definition, a point $gB$ belongs to  $U_e$ if $g$ is the homothety class of a matrix $M$ whose principal minors are different from zero. Equivalently (by Gauss reduction), it can be written as the product $M=LU$ of a lower unipotent triangular matrix $L$ and an upper triangular matrix $U\in B$ so, denoting the corresponding homothety classes by $\ell, b\in \PGL(4)$, we have that:
	$$gB=\ell b B=\ell B.$$
	Let us write $L$ as: 
	$$
	L=\begin{pmatrix}
		1 & 0 & 0 & 0 \\ 
		y_{11} & 1 & 0 & 0 \\
		y_{12}  & y_{22}  & 1 & 0 \\
		y_{13}  & y_{23}  & y_{33} & 1 \\
	\end{pmatrix}.
	$$
	Consider now a $1$-parameter subgroup $\mu\in \Mo(H)^*$, satisfying that $\mu(\alpha_{i,i})<0$ for every $i$ (and so $\mu(\alpha_{i,j})<0$, for every $i,j$). Given $h\in H$, we will have:
	$$
	\mu(h)(g B)=\mu(h)(\ell B)=
	\left[\begin{pmatrix}
		1 & 0 & \ldots & 0 \\ 
		t^{-\mu(\alpha_{1,1})}y_{11} & 1 & 0 & 0 \\
		t^{-\mu(\alpha_{1,2})}y_{12}  & t^{-\mu(\alpha_{2,2})}y_{22}  & 1 & 0 \\
		t^{-\mu(\alpha_{1,3})}y_{13} & t^{-\mu(\alpha_{2,3})}y_{23} & t^{-\mu(\alpha_{3,3})}y_{33} & 1 \\
	\end{pmatrix}\right]B.
	$$ 
	So we may conclude that $\lim_{t\to 0} t^{\mu}(gB)=eB$ and, in particular, that $eB\in \overline{H\cdot gB}$. 
	
	For the converse, note that $F\setminus U_e$ is an $H$-invariant closed subset; then, if $gB\in F\setminus U_e$, we have $\overline{H\cdot gB}\subset F\setminus U_e$ so, in particular, $eB\notin \overline{H\cdot gB}$.
\end{proof}

\subsection{Fundamental subtori and their associated $H$-invariant subsets}\label{ssec:fundasub}

Let us consider the set  $\OP_4$ of {\em ordered set partitions} of $\{0,1,2,3\}$ introduced in Definition \ref{def:fundamsubt}. It consists of $75$ elements; the trivial partition, $14$ ordered partitions with two elements, $36$ with $3$ elements and $24$ with $4$ elements. The natural action of $W=S_4$ has $8$ orbits on $\OP_4$ with, respectively, $1,4,6,4,12,12,12,24$ elements.

For our purposes we will only need to use partitions of the following types: 
\begin{itemize}[]
\item[($A_i$)] partitions $(S_0=\{i\},\,\, S_1=\{0,1,2,3\}\setminus\{i\})$,  ($0\leq i\leq 3$);
\item[($B_i$)] partitions  $(S_0=\{0,1,2,3\}\setminus\{i\},\,\, S_1=\{i\})$ ($0\leq i\leq 3$);
\item[($C_{ij}$)] partitions  $(S_0=\{i,j\},\,\, S_1=\{0,1,2,3\}\setminus\{i,j\})$ ($0\leq i<j\leq 3$);
\item[($D_{ij}$)] partitions $(S_0=\{i\},\, S_1=\{0,1,2,3\}\setminus\{i,j\},\,S_2=\{j\})$ ($0\leq i,j\leq 3$). 
\end{itemize}

Given a nontrivial partition $\phi\in\OP_4$, we consider the action on $F$ of the subtorus $T(\phi)$. One may easily check that the fixed point components of these action are flag manifolds; for instance, for partitions of types $A,B,C,D$ the fixed point components are rational homogeneous varieties (complete flag varieties of subspaces of $\C^3$ for the types $A,B$, $\P^1\times\P^1$ for type $C$ and $\P^1$ for the type $D$).

To each fixed point component $Y$ of the $T(\phi)$-action we will associate a closed subset of $F$, as follows:

\begin{definition}\label{def:fundHinv}
Let $\phi$ be a nontrivial ordered partition $\phi\in \OP_4$, consisting of $k+1$ subsets $S_0,\dots,S_k$, and let $R(\phi)$ be the set of group homomorphisms $\rho:\Mo(\phi)\to\Z$ defining nonnegative $1$-parametric subgroups of $T(\phi)$ (see Definition \ref{def:pos1param}). 
Let $Y$ be a fixed point component of the $T(\phi)$-action on $F$. We define:
$$
\displaystyle F^+_\phi(Y):=\ol{\left\{x\in F\left|\,\,\lim_{t\to 0} (t^{\rho\circ\mu_\phi}\,x)\in Y,\,\, \mbox{ for every }\rho \in R(\phi)\right.\right\}}.
$$
\end{definition}

It is straightforward to check that every $F^+_\phi(Y)$ is $H$-invariant for every $\phi$ and $Y$.
Note also that if $\phi$ is a partition with $2$ elements (types $A,B,C$), then $F^+_\phi(Y)$ is a Bia{\l}ynicki-Birula cell of $F$ with respect to some (fundamental) $\C^*$-action.

\section{Affine charts and their quotients}\label{sec:local}

We keep the notation and the assumptions of Setup \ref{setup:rest}. Corollary \ref{cor:momentgen} provides a description of the general cycle parametrized by the normalized Chow quotient $X$ of $F$. The remaining cycles are known to be unions of toric $H$-varieties whose moment polytopes are given by subdivisions of $P$ (see for instance \cite[Proposition 3.3]{Hu} or \cite[Proposition 1.2.11]{Kapranov}). We will see that the irreducible components of these cycles are the closures of orbits of points in subsets $F^+_\phi(Y)\subset F$ (see Definition \ref{def:fundHinv}), 
by analyzing the $H$-action on certain affine $H$-invariant charts $F_\sigma\subset F$, for $\sigma\in W$, namely the images by elements of $W$ of the open Bruhat cell of $F$. 
The main tool we will use is the concept of combinatorial quotient of a toric variety by the action of a subtorus (see Section \ref{ssec:prelimtorus}), and the fact that in our case the combinatorial quotients of the $H$-invariant charts are projective varieties (see Proposition \ref{prop:combquotisproj}), which are contractions of $X$. 

\subsection{$H$-invariant charts}\label{ssec:charts}

Recall that we have a natural $H$-equivariant isomorphism $T_{F,eB}\simeq \fg/\fb\simeq \fn$, which is the Lie algebra of nilpotent lower triangular matrices. Moreover $\fn$ is isomorphic to an open neighborhood of $eB$ in $F$, via the exponential map. Indeed, the morphism:
$$
\ex:\fn \to F,\quad \ex(g):=e^gB,
$$
is an open immersion (see, for instance, \cite[Proposition~1.2]{Huy}).
We define 
$$
F_e:=\ex(\fn)\subset F, \quad\mbox{and}\quad F_\sigma:=\sigma(F_e) \mbox{ for every $\sigma\in  W$.} 
$$
The subset $F_\sigma\subset F$ is an open neighborhood of $\sigma B$, for every $\sigma\in  W$. Furthermore, $F_\sigma$ is $H$-invariant; this is due to the fact that the exponential map commutes with conjugation, and that $\fn\subset \fg$ is $H$-invariant. Note that the map $\sigma:F_e\to F_\sigma$ is not $H$-equivariant, but it satisfies the property that for every $h\in H$ the following diagram is commutative :
$$
\xymatrix@C=35pt@R=20pt{F_e\ar[r]^h\ar[d]_{\sigma} &F_e\ar[d]^{\sigma}\\F_{\sigma}\ar[r]_{\conj_\sigma(h)}&F_\sigma}
$$
On the other hand, the inclusions of $F_e$ and $F_\sigma$ into $F$ provide an $H$-equivariant birational map:
\begin{equation}\label{eq:bir}
\varphi_\sigma:F_e\dashrightarrow F_\sigma.
\end{equation}
In the next section we will describe the combinatorial quotient of every $F_\sigma$ by the action of $H$. Obviously, thanks to the above $H$-equivariant isomorphisms it is enough to consider the case $\sigma=e$.

\subsection{The combinatorial quotient of the Lie algebra of upper triangular matrices}\label{sec:nilpotent}

In this section we  compute the quotient fan of the $H$-action on the affine chart $F_e\simeq \fn \simeq \C^6$. In analogy with Setup \ref{setup:rest}, we will denote the combinatorial quotient of $F_e$ as $X_e$.

We consider $F_e$ as an affine toric variety, with acting torus $T\simeq (\C^*)^6$, and whose associated cone is the positive orthant $\delta$ in $\No(T)_{\R}\simeq \fn(\R)$, which is the real vector space of nilpotent lower triangular matrices. We denote by $\Sigma(\delta)$ the fan of faces of $\delta$. The rays of the cone $\delta$ correspond to matrices $R_{ab}$, $1\leq a\leq b \leq 3$, with

\[ (R_{ab})_{i,j} =\begin{cases} 1 & \mbox{if } i=b+1, j=a\\ 0 & \mbox{otwerwise.} \end{cases} \]

The $H$-action on $F_e$ is given by conjugation, and, setting $t_i:= h_{i-1}h_{i}^{-1}$, for $1\leq i \leq 3$ and $h=(h_0,\ldots,h_3)\in H$, we obtain

\[ h\cdot 
\begin{pmatrix}
	0 & 0 & 0 & 0 \\
	y_{11} & 0 & 0 & 0 \\
	y_{12} & y_{22} & 0 & 0 \\
	y_{13} & y_{23} & y_{33} & 0 \\
\end{pmatrix} = 
\begin{pmatrix}
	0 & 0 & 0 & 0 \\
	t_1 y_{11} & 0 & 0 & 0 \\
	t_1t_2 y_{12} & t_2y_{22} & 0 & 0 \\
	t_1t_2t_3 y_{13} & t_2t_3 y_{23} & t_3 y_{33} & 0 \\
\end{pmatrix}
\]
Therefore, the weight map $\Mo(T)\to \Mo(H)$ is represented by the weight matrix 
\[
\begin{pmatrix} 
	1 & 0 & 0  & 1 & 0 & 1 \\
	0 & 1 & 0 & 1 & 1 & 1 \\
	0 & 0 & 1 & 0 & 1 & 1 \\
\end{pmatrix}
\]
and dually the action is represented by the inclusion $\No(H)\hookrightarrow \No(T)$ given by the transpose of the weight matrix. Let $\pi: \No(T)\to \No(T/H)$ be its cokernel,  
that under an appropriate identification of $\No(T/H)$ with $\Z^3$ is represented by the matrix:
\[
M(\pi)=\begin{pmatrix} 
-1 & -1 & 0 & 1 & 0 & 0 \\
 0 &-1 & -1 & 0 & 1 & 0 \\
-1 & -1 & -1 & 0 & 0 & 1 \\
\end{pmatrix}
\]

 We have computed the quotient fan $\Sigma$ of $\Sigma(\delta)$ (see  Section \ref{ssec:prelimtorus}) aided by a {\tt SageMath} script, and the outcome is the following:

\begin{proposition}\label{prop:combquotisproj}
	The combinatorial quotient $X_e$ is a smooth projective toric variety $X(\Sigma)$ of dimension $3$ and Picard number $4$, whose fan has the following rays in $\No(T/H)_\R=\R^3$: 
	\[
	\begin{array}{c}
		\rho_0=(-1,0,-1),\quad \rho_1 = (-1,-1,-1), \quad \rho_2=(0,-1,-1),\\[3pt] \rho_3=(1,0,0),\quad \rho_4=(0,1,0), \quad \rho_5=(0,0,1), \quad \rho_6=(0,0,-1).
	\end{array}
	\]
\end{proposition}

In other words, the rays of the quotient fan are the columns of the matrix $M(\pi)$, together with $(0,0,-1)$, which can be obtained  
by considering the intersection of the images of the cones $\langle R_{11},R_{12}\rangle$ and $\langle R_{33}, R_{23}\rangle$.

\subsection{Geometric description of $X_e$}\label{ssec:geomdescr} 

We will present here a geometric construction of the variety $X_e$, as a birational modification of the projective space $\P^3$, as well as a description of its $T/H$-invariant prime divisors.

\begin{construction}\label{const:1}

Consider the $3$-dimensional projective space $\P^3$ with coordinates $[y_0:y_1:y_2:y_3]$, endowed with the standard action of $T/H=(\C^*)^3$, $(t_1,t_2,t_3)[y_0:y_1:y_2:y_3]=[y_0:t_1y_1:t_2y_2:t_3y_3]$;  let $P$ be the point $[0:0:0:1]$ and let $\ell_1,\ell_2$ be the lines $\ell_1:y_0=y_2=0$, $\ell_2:y_0=y_1=0$.
We may construct $X_e$ starting from $\P^3$ by the following sequence of ($T/H$-equivariant) smooth blowups along rational curves:
\begin{itemize}
\item[($p_a$)] blow up $\P^3$ along the line $\ell_1$;
\item[($p_b$)] blow up the resulting variety along the strict transform of the line $\ell_2$;
\item[($p_d$)] blow up the resulting variety along the strict transform of $(p_a)^{-1}(P)$.
\end{itemize}
In the sequel, we will denote by $p$ the composition of these blowups:
\begin{equation}\label{eq:projP3}
p:=p_d\circ p_b\circ p_a:X\to \P^3.
\end{equation}
\end{construction}

\subsubsection*{Torus invariant divisors on the variety $X_e$}

Let us introduce the following notation for some $T/H$-invariant prime divisors in $X_e$, defined by rays of $\pi(\Sigma(\delta))$:

\begin{table}[h!!]\label{tab:divxe}
\begingroup
\renewcommand*{\arraystretch}{1.1}
\begin{tabular}{|c||c|l|}
\hline
Divisor&Ray&Description\\\hline\hline
$A_1$&$\rho_0$&proper transform of the exceptional divisor of $p_a$\\\hline
$B_2$&$\rho_2$&proper transform of the exceptional divisor of $p_b$\\\hline
$C_{02}$&$\rho_1$&proper transform of the plane $y_0=0$\\\hline
$D_{12}$&$\rho_6$&exceptional divisor of $p_d$\\\hline
\end{tabular}\par\medskip
\caption{$T/H$-invariant prime divisors in $X_e$ generating $\Pic(X_e)$.}
\vspace{-20pt}
\endgroup
\end{table}

\begin{remark}\label{rem:picXe}
Besides these $4$ divisors, we have $3$ more $T/H$-invariant divisors in $X_e$, that we denote by $E,F,G$, corresponding to the rays $\rho_4,\rho_3,\rho_5$, respectively. One may then check that  we have the following linear equivalence relations in $X_e$:
$$
E-B_2-C_{02}\equiv F-A_1-C_{02}\equiv G-A_1-B_2-C_{02}-D_{12}\equiv 0.
$$
In particular, the Picard group of $X_e$ is generated by the classes of 
$A_1,B_2,C_{02},D_{12}$.
\end{remark}

The divisor $5C_{02}+3A_1+3B_2+2D_{12}$ is ample and, using the functions presented in \cite{BOSSage}, we have represented in Figure \ref{fig:ampleXe} the corresponding ample polytope.  
It is a $3$-dimensional lattice polytope with $f$-vector equal to $(10,15,7)$.  Following Construction \ref{const:1}, one may easily obtain it from a tetrahedron by means of three successive truncations of edges.

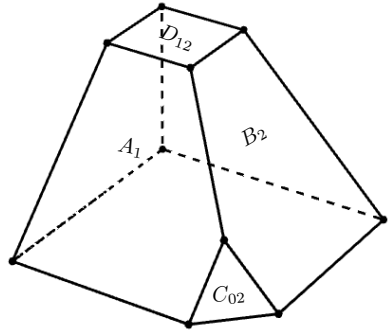
\begin{figure}[h!!]
	\centering
	\resizebox{0.43\textwidth}{!}{%
		\begin{tikzpicture}
			\node at (0, 0) {\includegraphics[width=0.5\textwidth]{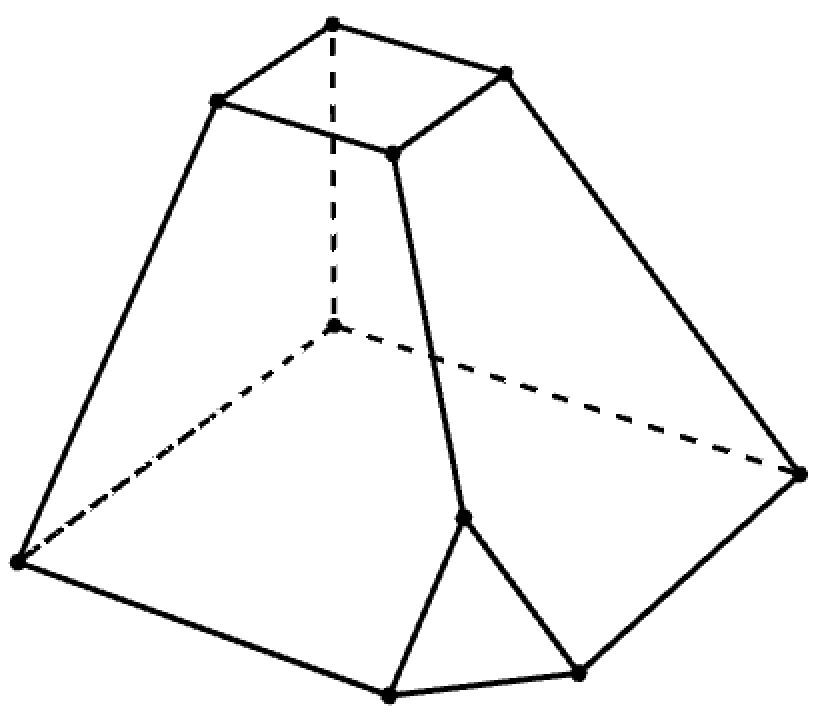}};
			\node[rotate=-10] at (-1.1, 0.3) {$A_1$};
			\node[rotate=30] at (0.9, 0.5) {$B_2$};
			\node[rotate=5] at (0.5, -2.1) {$C_{02}$};
			\node[rotate=-19] at (-0.35, 2.1) {$D_{12}$};
		\end{tikzpicture}  
	}
	\caption{An ample polytope of $X_e$ and some of its boundary divisors.} \label{fig:ampleXe}
\end{figure}

\subsection{Boundary divisors and fundamental subtori} \label{ssec:boundarydiv}

The notation we have introduced in Table \ref{tab:divxe} for some $T/H$-invariant divisors is meant to resemble the one we have used in Section \ref{ssec:fundasub} for some particular fundamental subtori. Let us now explain, case by case, the relation between partitions of types $A,B,C,D$ and the prime divisors of types $A,B,C,D$, above. The key idea we will use here is the following:

\begin{remark}\label{rem:rel}
Let us consider two faces $\delta_1,\delta_2$ of the orthant $\delta$, with associated $T$-orbits $O(\delta_1),O(\delta_2)\subset\fn$, and assume that 
$\dim(\pi(\delta_i))=\dim(\delta_i)$, $i=1,2$. This assumption means that the action of $H$ on $O(\delta_i)$ has orbits of dimension three, and the corresponding (geometric) quotient of $O(\delta_i)$ by $H$ is the $T/H$-orbit $O(\pi(\delta_i))$, for $i=1,2$. If $\delta_1$ is relevant and $\delta_2$ is its relevant companion (see Definition \ref {def:rel}), then the sum $\delta_1+\delta_2$ has dimension larger than $\dim(\pi(\delta_1+\delta_2))$, and so the $H$-orbits in $O(\delta_1+\delta_2)$ have dimension smaller than three. Since $O(\delta_1+\delta_2)$ is contained in $\ol{O(\delta_1)}\cap \ol{O(\delta_2)}$, it then follows that each element in $O(\pi(\delta_1)\cap \pi(\delta_2))$ corresponds to an $H$-orbit in $O(\delta_1)$ and an $H$-orbit in $O(\delta_2)$, with a common border in an $H$-orbit contained in $O(\delta_1+\delta_2)$.
\end{remark}

\medskip

\noindent{\bf (Type $\mathbf {A_1}$).} Let us consider the $\C^*$-actions associated to the partitions $A_1$ and $B_1$ on the nilpotent Lie algebra $\fn$, and represent their weights in the following lower triangular matrices:
\[
\begin{pmatrix}
0&0&0&0\\
\fbox{\,1\,}&0&0&0\\
\fbox{\,0\,}&\fbox{-1}&0&0\\
\fbox{\,0\,}&\fbox{-1}&\fbox{\,0\,}&0
\end{pmatrix}, \quad
\begin{pmatrix}
0&0&0&0\\
\fbox{-1}&0&0&0\\
\fbox{\,0\,}&\fbox{\,1\,}&0&0\\
\fbox{\,0\,}&\fbox{\,1\,}&\fbox{\,0\,}&0
\end{pmatrix}
\] 
One can then see that they have the same fixed point component passing by $eB$, that we denote by $Y_e$, and that 
$F_{A_1}^+(Y_e)\cap F_e$, $F_{B_1}^+(Y_e)\cap F_e$ are the linear subspaces corresponding to the cones $\delta(A_1),\delta(B_1)\preccurlyeq \delta$ generated by $R_{22}$, $R_{23}$, and by $R_{11}$, respectively. Then we have that:
$$
\pi(\delta(A_1))=\langle \rho_1,\rho_4\rangle,\quad \pi(\delta(B_1))=\langle \rho_0\rangle,
$$
and in particular we see that 
$$\pi(\delta(B_1))\subset \pi(\delta(A_1)),$$
but it is not one of its faces. We then conclude that $\delta(A_1)$ is relevant, with relevant companion $\delta(B_1)$; then the prime divisor in $X_e$ corresponding to the ray $\pi(\delta(A_1))$ corresponds to a family of (intersections with $F_e$ of) $H$-invariant cycles having an irreducible component in $F_{A_1}^+(Y_e)$, and one in $F_{B_1}^+(Y_e)$ (see Remark \ref{rem:rel}). Note finally that the divisor associated to $\pi(\delta(A_1))\cap \pi(\delta(B_1))=\langle \rho_0\rangle$ is the one we have denoted by $A_1\subset X_e$. 

\medskip

\noindent{\bf (Type $\mathbf {B_2}$).} An argument analogous to the one above shows that the fundamental actions of type $B_2,A_2$ have the same fixed point component passing by $eB$ (abusing notation, we denote it again by $Y_e$), that the corresponding fundamental invariant subsets satisfy:
$$
\pi(\delta(B_2))=\langle \rho_1,\rho_3\rangle, \quad \pi(\delta(A_2))=\langle \rho_2\rangle, 
$$
and that 
$$\pi(\delta(A_2))\subset \pi(\delta(B_2)),\quad\mbox{but }\pi(\delta(A_2))\not\preccurlyeq \pi(\delta(B_2)),$$
that is, $\delta(B_2)$ is relevant, with relevant companion $\delta(A_2)$. 
The divisor in $X_e$ associated to the ray $\pi(\delta(B_2))\cap \pi(\delta(A_2))=\langle \rho_2\rangle$, that we denoted by $B_2$, corresponds to a family of (intersections with $F_e$ of) $H$-invariant cycles having an irreducible component in $F_{B_2}^+(Y_e)$, and one in $F_{A_2}^+(Y_e)$.

\begin{figure}[h!!]
\includegraphics[width=5cm]{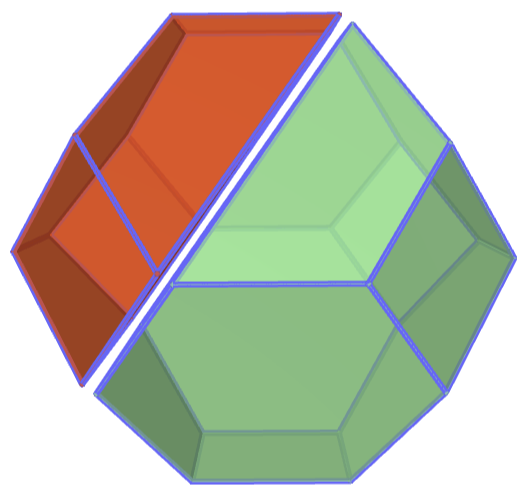} \qquad \includegraphics[width=5cm]{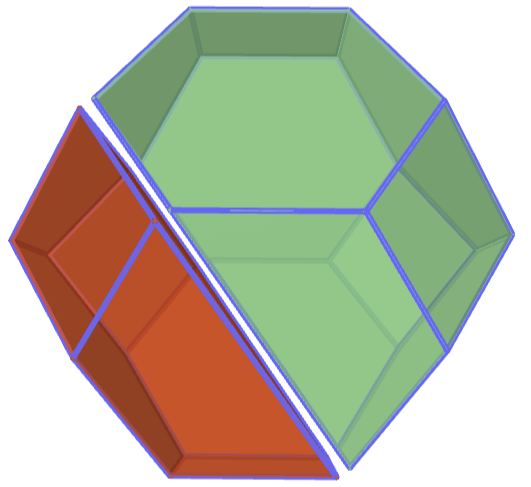}
\caption{Subdivisions of $P$ associated to the divisors $A_1,B_2$.\label{fig:TypeAB}}
\end{figure}

\medskip

\noindent{\bf (Type $\mathbf {C_{02}}$).} In this case, we consider the partitions $C_{02}$, $C_{13}$, that have the same component passing by $eB$, denoted again by $Y_e$. The corresponding fundamental invariant subsets satisfy:
$$
\pi(\delta(C_{02}))=\langle \rho_0,\rho_2,\rho_5\rangle, \quad \pi(\delta(C_{13}))=\langle \rho_1\rangle, 
$$
and so 
$$\pi(\delta(C_{13}))\subset \pi(\delta(C_{02})), \quad\mbox{and }\pi(\delta(C_{13}))\not\preccurlyeq\pi(\delta(C_{02})).$$
The divisor in $X_e$ associated to the ray $\langle \rho_1\rangle$, that we denoted by $C_{02}\subset X_e$  corresponds to a family of (intersections with $F_e$ of) $H$-invariant cycles having an irreducible component in $F_{C_{02}}^+(Y_e)$, and one in $F_{C_{13}}^+(Y_e)$.

\medskip

\noindent{\bf (Type $\mathbf {D_{12}}$).} This last case is slightly different from the others, since it involves the action of the subtorus of rank two of $H$ associated to the partition $D_{12}$, that has rank two; we will consider it together with the partitions $C_{12}$ and $C_{03}$. The fixed point component $Y_e$ passing by $eB$ of the action associated to $D_{12}$ is a $\P^1$, while the action associated to $C_{12}$ and $C_{03}$ have the same fixed point component $Y'_e$, isomorphic to $\P^1\times \P^1$ and containing $Y_e$. Moreover, the orbit-cone correspondence reads here as:
$$\begin{array}{c}
F_{D_{12}}^+(Y_e)\cap F_e\leftrightarrow\langle R_{12},R_{22},R_{23}\rangle, \\[3pt] F_{C_{12}}^+(Y'_e)\cap F_e\leftrightarrow\langle R_{23},R_{33}\rangle,\quad F_{C_{03}}^+(Y'_e)\cap F_e \leftrightarrow\langle R_{11},R_{12}\rangle.
\end{array}
$$ 
Projecting this cones via $\pi$ we get three cones
$$
\pi(\delta(D_{12}))=\langle \rho_3,\rho_4,\rho_1\rangle,\quad 
\pi(\delta(C_{12}))=\langle \rho_2,\rho_4\rangle,\quad
\pi(\delta(C_{03}))=\langle \rho_0,\rho_3\rangle.
$$
intersecting effectively in the ray generated by $\rho_6$. We get that $\delta(D_{12})$ is relevant, and that $\delta(C_{12})$, $\delta(C_{03})$ are two relevant companions of it. The prime divisor defined by the ray generated by $\rho_6$, that we denoted by $D_{12}\subset X_e$, corresponds then to a family of (intersections with $F_e$ of) $H$-invariant cycles having an irreducible component in $F_{D{12}}^+(Y_e)$,  $F_{C_{12}}^+(Y_e)$, and  $F_{C_{03}}^+(Y_e)$. Note that $\langle\rho_6\rangle$ can be written simply as $\pi(\delta(C_{12}))\cap \pi(\delta(C_{03}))$, so the divisor $D_{12}$ is completely determined by the two sets $F_{C_{12}}^+(Y_e)$,  $F_{C_{03}}^+(Y_e)$; note also that $\langle\rho_6\rangle$ is the only ray of the quotient fan $\pi(\Sigma(\delta))$ that is not the projection of a ray of $\delta$.

\begin{figure}[h!!]
\includegraphics[width=5cm]{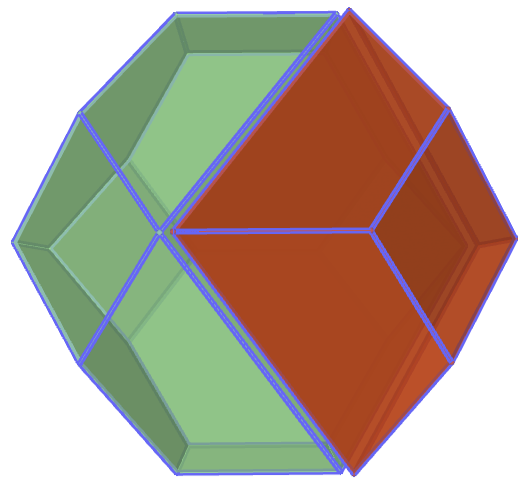} \qquad \includegraphics[width=5cm]{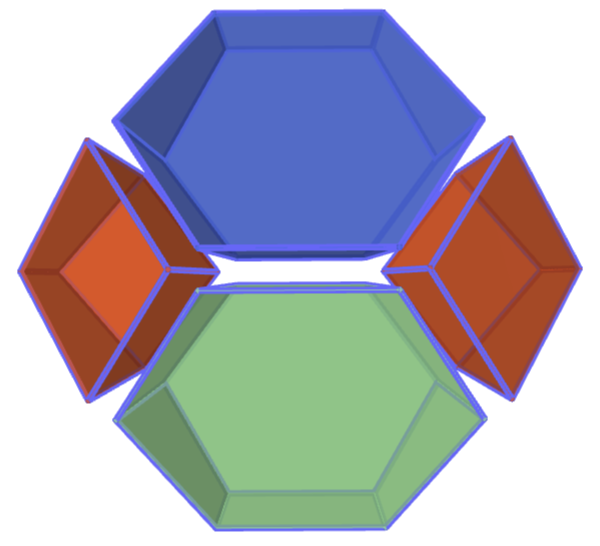}
\caption{Subdivisions of $P$ associated to the divisors $C_{02},D_{12}$.\label{fig:TypeCD}}
\end{figure}

\subsection{Description of $X_e$ as an inverse limit of GIT quotients}\label{ssec:combislimit} 

We have defined the variety $X_e$ as a toric variety associated to the quotient fan $\Sigma$ of the fan $\Sigma(\delta)$. We will now relate it to some GIT quotients of $F_e$ by the action of $H$.
Let us start by introducing the following notation  for the ray generators  of the orthant $\delta$: 
$$E_0=R_{11},\quad E_1=R_{22},\quad E_2=R_{33},\quad E_3=R_{21},\quad E_4=R_{32},\quad E_5=R_{31},$$ 
and the following notation for the positive dimensional faces of $\delta$:
$$
\sigma_{(i_1\dots i_k)}:=\R_{\geq 0}(E_{i_1},\dots, E_{i_k}), \quad \mbox{for}\quad  0\leq i_1<\dots<i_k\leq 5.
$$
Given a face $\sigma_{(i_1\dots i_k)}$ of $\delta$, we denote by $O_{(i_1\dots i_k)}\subset F_e$ the corresponding $T$-orbit, and by $Z_{(i_1\dots i_k)}$ its closure, that is:
$$
Z_{(i_1\dots i_k)} = \bigcup_{\{j_1,\dots, j_r\}\supset\{i_1,\dots,i_k\}} O_{(j_1\dots j_r)}.
$$
We consider the following subfans of $\Sigma(\delta)$:
$$
\begin{array}{l}
\Sigma'_0=\left\{\tau\in\Sigma(\delta)\ |\ \sigma_{(03)},\sigma_{(13)},\sigma_{(14)},\sigma_{(24)},\sigma_{(025)}\not\preccurlyeq \tau\right\},\\[3pt]
\Sigma'_+=\left\{\tau\in\Sigma(\delta)\ |\ \sigma_{(03)},\sigma_{(13)},\sigma_{(14)},\sigma_{(025)},\sigma_{(245)}\not\preccurlyeq \tau\right\},\\[3pt]
\Sigma'_-=\left\{\tau\in\Sigma(\delta)\ |\ \sigma_{(13)},\sigma_{(14)},\sigma_{(24)},\sigma_{(025)},\sigma_{(035)}\not\preccurlyeq \tau\right\}.\\[3pt]
\end{array}
$$
Let us denote by $U_0,U_+,U_-\subset F_e$ the corresponding $T$-invariant open subsets of $F_e$.
It is a straightforward computation to check the following:
\begin{itemize}[leftmargin=\xx pt]
\item The fans $\Sigma'_+,\Sigma'_-,\Sigma'_0$ project bijectively via the quotient map $\pi$ to fans  $\Sigma_+,\Sigma_-,\Sigma_0$; equivalently, we have three toric morphisms: $$U_+\to U_+/H=X(\Sigma_+),\quad U_-\to U_-/H=X(\Sigma_-),\quad U_0\to U_0/H=X(\Sigma_0).$$
\item Since every cone of $\Sigma'_+$ (resp. $\Sigma'_-$) is contained in $\Sigma'_0$, we get $T$-equivariant morphisms $U_+\lra U_0\longleftarrow U_-$ 
that descend to $T/H$-equivariant morphisms:
$$
\xymatrix{X(\Sigma_+)\ar[r]&X(\Sigma_0)&X(\Sigma_-)\ar[l]}
$$
 
\item The difference between these varieties is that the closed sets $\pi(Z_{(24)})\subset X(\Sigma_+)$, $\pi(Z_{(03)})\subset X(\Sigma_-)$ are contracted  to the point $\pi(Z_{(0234)})$ in $X(\Sigma_0)$. The common refinement of $\Sigma_+$, $\Sigma_-$, whose associated toric variety resolves the induced birational map $X(\Sigma_+)\dashrightarrow X(\Sigma_-)$, is the quotient fan $\Sigma$. One may check 
that the map $X(\Sigma_+)\dashrightarrow X(\Sigma_-)$ is locally the standard Atiyah flip (cf. \cite[Example~11.1.12]{CLS}, \cite[\S~1.3]{ReidFlip}),  
and that we have a Cartesian square: 
$$
\xymatrix@C=10pt@R=15pt{&X(\Sigma)\ar[rd]\ar[ld]&\\X(\Sigma_+)\ar[rd]&&X(\Sigma_-)\ar[ld]\\&X(\Sigma_0)&
}
$$
\item The exceptional locus of the two contractions $X(\Sigma)\to X(\Sigma_\pm)$ is precisely the divisor $D_{12}$, which is isomorphic to $\P^1\times \P^1$.
\end{itemize}
Figure \ref{fig:flip}  represents moment polytopes of the above toric varieties. 
\begin{figure}[ht]
  \centering
  \resizebox{0.75\textwidth}{!}{%
\begin{tikzpicture}
  \node at (0, 2.7) (1) {\includegraphics[width=0.25\textwidth]{quot_TO}};
  \node at (-4.5, 0) (2) {\includegraphics[width=0.22\textwidth]{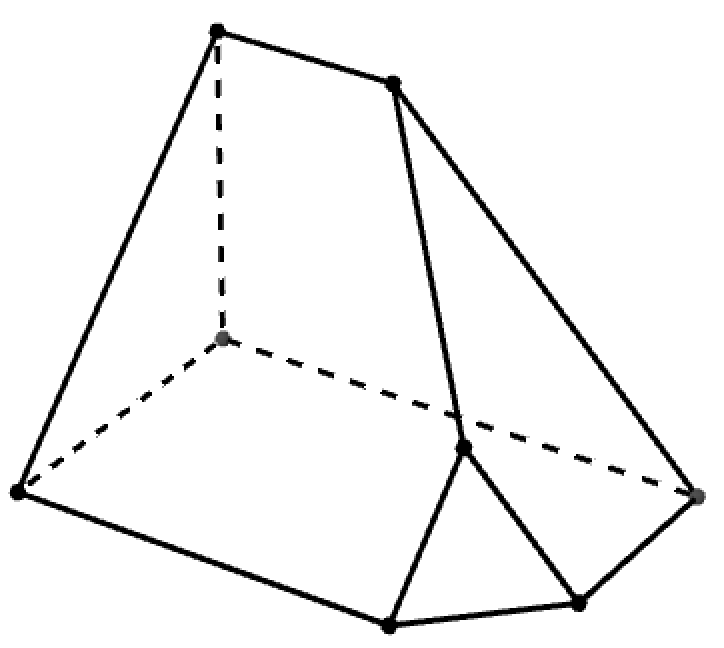}};
  \node at (4.5, 0) (3) {\includegraphics[width=0.22\textwidth]{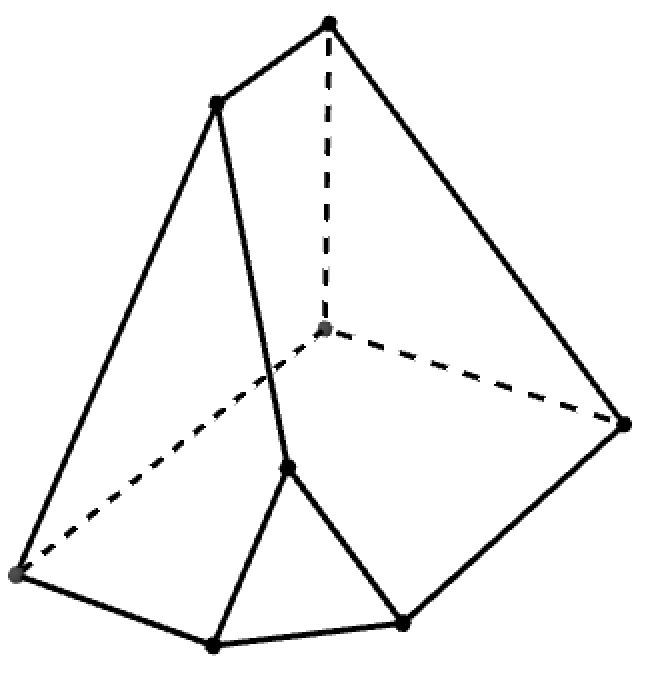}};
  \node at (0, -2.7) (4) {\includegraphics[width=0.20\textwidth]{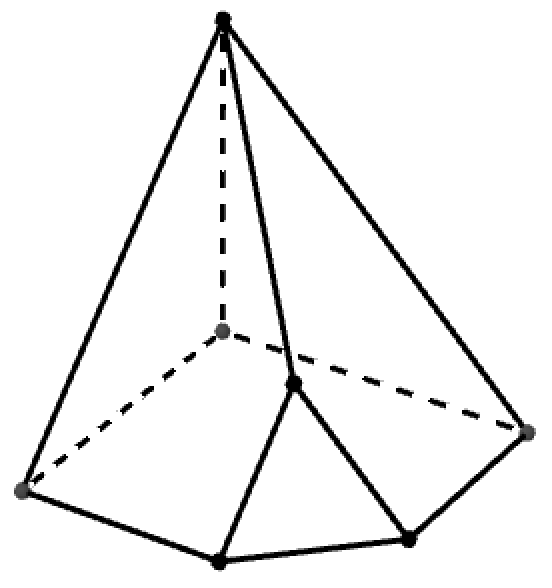}};
   \draw[->] (1) -- node {} (2);
   \draw[->] (1) -- node {} (3);
   \draw[->] (2) -- node {} (4);
   \draw[->] (3) -- node {} (4);
\end{tikzpicture}  
}
\caption{The toric quotient of $F_e$ as inverse limit of three GIT quotients of $F$ (two of them geometric).\label{fig:flip}}
\end{figure}

\begin{remark}\label{rem:A1B2D12} 
Notice that, with the notation of Section \ref{ssec:geomdescr}, the variety $X(\Sigma_+)$ is the blowup of $\P^3$ along the images of $B_2$, and, subsequently, $A_1$, while  $X(\Sigma_-)$ is obtained by blowing up the image of $A_1$, and then the image of $B_2$. The resolution $X(\Sigma)$ of the induced birational map from $X(\Sigma_+)$ to $X(\Sigma_-)$ is then obtained by blowing up the image of $D_{12}$ in the resulting varieties. On the other hand, the variety $X(\Sigma_0)$ can be obtained by blowing up $\P^3$ along the reducible conic  defined as the union of the images of $A_1$ of $B_2$, and $X(\Sigma)$ can be obtained out of $X(\Sigma_0)$ by blowing up its singular point. Finally, we note that a standard computation provides the nef cone of $X(\Sigma)$: it is simplicial, of dimension $4$, and its extremal rays correspond to two contractions to $\P^1$, a contraction to the quadric cone, and the contraction to $\P^3$ upon which we have constructed $X(\Sigma)$.
\end{remark}

\begin{lemma}
	The subsets $U_+,U_-$ (resp. $U_0$) are nonempty open subsets of stable (resp. semistable) points of $F$ with respect to  linearizations of the $H$-action on $F$.
\end{lemma}

\begin{proof}
	Let us prove the result for $U_0$; the other cases are similar. Thanks to \cite[Converse~1.12,~1.13]{MFK}, it is enough to show that $\pi: U_0\to X(\Sigma_0)$ is an affine morphism. Since $\Sigma'_0$ projects bijectively to $\Sigma_0$, we have that, for any $\sigma\in\Sigma_0$, $\pi^{-1}(F_\sigma)=F_{\pi^{-1}(\sigma)}=\bigcup_{\tau\preccurlyeq \pi^{-1}(\sigma)} \cO(\tau)$ is an affine open subset of $U_0$.
\end{proof}

In particular $X(\Sigma_+)$, $X(\Sigma_-)$, $X(\Sigma_0)$ are GIT quotients of $F$, 
and $X_e$ is their inverse limit. Then the natural maps from ${\mathcal L}F$ (which is equal to $\CF$ by Proposition \ref{prop:BHR} ) to $X(\Sigma_+)$, $X(\Sigma_-)$, $X(\Sigma_0)$ factor via $X_e$, and we get a (clearly birational) morphism $\pi_e:X\to X_e$. 

The same arguments provide a similar description of the combinatorial quotient of $F_\sigma$ for every $\sigma\in W$, that we denote by $X_\sigma$, and of its boundary divisors $A_{\sigma(1)}$,  $B_{\sigma(2)}$, $C_{\sigma(0)\sigma(2)}$, $D_{\sigma(1)\sigma(2)}\subset X_\sigma$. Furthermore we get a birational morphism $\pi_\sigma: X\to X_\sigma$, for every $\sigma\in W$. Note that $X$ is, by definition, normal and all the $X_\sigma$'s are smooth, hence we conclude the following:

\begin{corollary}\label{cor:birmor}
For every $\sigma\in W$ we have birational contractions:
\[
\pi_\sigma:X\lra X_\sigma.
\]
\end{corollary}

Let us denote by $\ol{X}$ the normalization of the image of the product morphism:
\[
(\pi_\sigma)_{\sigma\in W}:X \lra \prod_{\sigma\in W}X_\sigma,
\]
and by $\ol{\pi}:X\to \ol{X}$ the induced morphism. We claim the following:

\begin{proposition}\label{prop:invlim}
The morphism $\ol{\pi}:X\to \ol{X}$ is an isomorphism.
\end{proposition}

\begin{proof}
Since $\ol{\pi}$ is, by construction, birational, and $\ol{X}$ is normal, then it is enough to show that the map is finite. Assume, for the sake of contradiction, that it contracts an irreducible curve $C\subset X$. Equivalently, $C$ gets contracted by every $\pi_\sigma:X\to X_\sigma$, and so all the $H$-invariant cycles parametrized by points of $C$ have the same intersection with $F_\sigma$, for $\sigma\in W$. We get to a contradiction by noting that the $F_\sigma$'s form an open covering of  the flag variety $F$.
\end{proof}

\begin{remark}\label{rem:irredcycles}
A consequence of the above discussion is that the closed subset of the Chow quotient $$\bigcup_{\sigma\in W}\pi_{\sigma}^{-1}(A_{\sigma(1)}\cup B_{\sigma(2)}\cup C_{\sigma(0)\sigma(2)}\cup D_{\sigma(1)\sigma(2)})\subset X$$ contains all the reducible $H$-invariant cycles of the universal family parametrized by $X$, and the  the different ways in which an $H$-invariant cycle may break, locally around fixed points $\sigma B$, $\sigma\in W$  are the ones described in Section \ref{ssec:boundarydiv}.
\end{remark}

\section{The tile group}\label{sec:tilegp}
\label{sec:merging}

The intersections of the affine charts $F_\sigma\subset F$ induce birational maps among their quotients $X_\sigma$.
Composing these with the isomorphisms $X_\sigma\to X_e$ induced by the translations $\sigma^{-1}:F_\sigma\to F_e$ yields a group of birational automorphisms of $X_e$, which can be enlarged considering also the automorphism of $X_e$ induced by the anti-transposition $\tau$. Considering the contraction $p: X_e\to \P^3$ described in (\ref{eq:projP3}) we obtain a group of birational automorphisms of $\P^3$ -- called {\em tile group} -- that will play a pivotal role in the description of the Chow quotient $X$ of $F$. 
In this section we will study this group, describe its generators explicitly, and use it to define a set of boundary divisors in $X$, that parametrize reducible $W$-invariant cycles in $F$. 

\subsection{Definition of the tile group}\label{ssec:defFoctah}

The natural action of the normalizer $\No_G(H)$ on $F$ descends to an action of $W$ on $X$; 
abusing notation, the automorphism of $X$ determined by an element $\sigma\in W$ will be denoted by $\sigma:X\to X$. Moreover the automorphism of $F$ induced by the anti-transposition map $\tau$ (see Section \ref{ssec:prelimPGL}); preserves $H$, hence it descends to an automorphism $\tau$ of $X$.

By construction, these automorphisms preserve the set of contractions $\pi_w$; denoting by $\phi_\sigma:X_w\to X_{\sigma w}$, for every $w,\sigma\in W$, the morphism induced by the translation $\sigma:F_w\to F_{\sigma w}$, we have a commutative diagram:
\begin{equation}\label{eq:phiw1}
\xymatrix{X\ar[r]^{\sigma}\ar[d]_{\pi_{w}}&X\ar[d]^{\pi_{\sigma w}}\\ X_w\ar[r]_{\phi_\sigma}&X_{\sigma w}}
\end{equation}
In the case of the anti-transposition $\tau$,  if $w=r_{i_1}r_{i_2}\dots r_{i_k}$, then $\tau$ sends $\pi_w$ to $\pi_{w'}$, where $w'=r_{n+1-i_k}\dots r_{n+1-i_2} r_{n+1-i_1}$. In particular, it sends the contraction $\pi_e$ to itself. Note that the induced map $\phi_\tau:X_e\to X_e$  is an isomorphism such that $\phi_\tau=\phi_\tau^{-1}$, but not the identity.

\begin{definition}\label{def:birgroup}
For every $\sigma\in W$ we consider the birational map $\pi_{e}\circ (\pi_\sigma)^{-1}: X_\sigma\dashrightarrow X_e$, and define a birational automorphism $\sigma$ of $X_e$ as follows:
\[
\xymatrix@C=45pt{X_e\ar[r]_{\phi_{\sigma}} \ar@/^10pt/@{-->}[rr]^{\sigma:=\pi_e\circ (\pi_\sigma)^{-1} \circ \phi_{\sigma}}
&X_\sigma \ar@{-->}[r]_{\pi_{e}\circ (\pi_\sigma)^{-1}}
&X_e}
\]
In this way we define a homomorphism of groups $W\to \Bir(X_e)$, that we may extend it to $W\rtimes \Z/2\Z$ by sending $\tau$ to the automorphism $\phi_\tau$. In other words, for every $\sigma\in W\rtimes \Z/2\Z$, $\sigma\in\Bir(X_e)$ is defined so that:
\begin{equation}\label{eq:defsigma}
\pi_\sigma\circ (\pi_e)^{-1}=\phi_{\sigma}\circ \sigma^{-1}.
\end{equation} 
\end{definition}

Furthermore, we may use Construction \ref{const:1}, that presents $X_e$ as a birational modification of $\P^3$, to produce a homomorphism of groups:
$$W\rtimes \Z/2\Z\to \Bir(\P^3).$$
Abusing notation, the birational transformation of $\P^3$ associated with an element $\sigma\in W\rtimes \Z/2\Z$ will also be denoted by $\sigma:\P^3\dashrightarrow \P^3$. 

\begin{definition}\label{def:FOgroup}
The image of $W\rtimes \Z/2\Z$ into  $\Bir(\P^3)$ will be denoted by
\[
\FO\subset \Bir(\P^3),
\]
and called the {\it tile group} 
-- the reason for this name will be clear by the end of the section (see Figure \ref{fig:20}).\end{definition}

We will now compute a set of generators of the group, and study its behavior with respect to the divisors of type $A,B,C,D$ introduced in Section \ref{ssec:geomdescr}.

\subsection{Generators of the group}\label{ssec:F-octah}
 
In order to study the tile group it is  enough to compute its generators, which are the images of $r_1,r_2,r_3,\tau\in \FO$ into $\Bir(\P^3)$.

\begin{construction}\label{const:birmaps}
We take a general element $x$ of $\fn$, i.e., the class of a matrix 
$$
Y=\begin{pmatrix}
	0 & 0 & 0 & 0 \\
	1 & 0 & 0 & 0 \\
	y_{1} & 1 & 0 & 0 \\
	y_{3} & y_{2} & 1 & 0 \\
\end{pmatrix},
$$ 
which corresponds to the point $[1:y_1:y_2:y_3]\in\P^3$.
 
We first compute the matrix $r_i e^Y$ (resp. $w_0((e^Y)^\top)^{-1}w_0$ in the case of the anti-transposition $\tau$, see the last item of Section \ref{ssec:prelimPGL}), and its LU-decomposition.

We then define the matrix $M$ as the logarithm of the unipotent low triangular factor $L$, whose  entries are rational functions in the variables $y_{i}$. 
We finally compute the class of $M$  modulo $H$, obtaining three rational functions $f_1,f_2,f_3$ in the variables $y_i$, interpret them as a rational map $\P^3\dashrightarrow\P^3$, sending $[1:y_1:y_2:y_3]$ to $[1:f_1:f_2:f_3]$, 
and homogenize these expressions, adding a new variable $y_0$.  
\end{construction}

We have performed the above procedure using {\tt SageMath}, obtaining expressions for the birational maps associated to $r_1,r_2,r_3,\tau$ in the homogeneous coordinates $[y_0:y_1:y_2:y_3]$.
These expressions take a particularly simple form after the following change of coordinates:
\begin{equation}
\begin{pmatrix}x_0\\x_1\\x_2\\x_3 \end{pmatrix}=\begin{pmatrix}
6 & 0 &  0 &  0\\
3 & -6 & 0 & 0\\
3  & 0 & 6 & 0\\
2 & -3 & 3 & -6
\end{pmatrix}\begin{pmatrix}y_0\\y_1\\y_2\\y_3 \end{pmatrix}
\label{eq:change}
\end{equation}
The outcome we obtain is the following:
\begin{proposition}\label{prop:birmaps}
The group $\FO$ is generated by the linear transformations $r_1$, $r_3$ and $\tau$, defined by the matrices
\[
R_1=
\begin{pmatrix}
0 & 1 & 0 & 0 \\
1 & 0 & 0 & 0 \\
0 & 0 & 0 & 1 \\
0 & 0 & 1 & 0
\end{pmatrix}
\qquad
R_3=
\begin{pmatrix}
0 & 0 & 1 & 0 \\
0 & 0 & 0 & 1 \\
1 & 0 & 0& 0 \\
0 & 1 & 0 &0
\end{pmatrix}
\qquad
T=
\begin{pmatrix}
1 & 0 & 0 & 0 \\
0 & 0& 1 & 0 \\
0 & 1 & 0 & 0 \\
0 & 0 & 0 & 1
\end{pmatrix}
\]
and by the quadratic Cremona transformation  $r_2:\PP^3 \dasharrow \PP^3$  given by 
\[
r_2([x_0:x_1:x_2:x_3])=[(x_1 - x_0)(x_2-x_0):   x_1(x_2-x_0) :x_2(x_1-x_0) :x_1x_2 - x_0x_3].
\]
\end{proposition}

In particular, one may check that all the generators are involutions, and the relations among them are those defining the octahedral symmetry group, that is:
\[
 r_1r_2r_1 =r_2r_1r_2, \qquad r_3r_2r_3 =r_2r_3r_2, \qquad 
\tau r_1 \tau = r_3 \qquad \tau r_2 \tau = r_2. 
\]
In particular, we get that the group $\FO$ is isomorphic to $W\rtimes \Z/2\Z$:
\begin{corollary}\label{cor:birmaps}
The homomorphism $W\rtimes \Z/2\Z\lra 
\Bir(\P^3)$ is injective.
\end{corollary}

\subsection{The boundary divisors}\label{ssec:Foctdiv}

\begin{definition}\label{def:alldivisors}
Given $\sigma\in W$, we consider the  
isomorphism $\phi_\sigma: X_e\to X_\sigma$, and the birational contraction $\pi_\sigma:X\to X_\sigma$. The proper transforms in $X$ via $\pi_\sigma$ of the image via $\phi_\sigma$ of the divisors  $A_1,B_2,C_{02},D_{12}$ will be denoted by $$A_{\sigma(1)},B_{\sigma(2)},C_{\sigma(0)\sigma(2)},D_{\sigma(1)\sigma(2)}\subset X.$$ 
Alternatively, we may define them as the proper transforms in $X$ of the boundary divisors $A_{\sigma(1)},B_{\sigma(2)},C_{\sigma(0)\sigma(2)},D_{\sigma(1)\sigma(2)}\subset X_\sigma$ (see Section \ref{ssec:combislimit}). We will call these prime divisors the {\em boundary divisors of  
$X$}. By definition, the group $W$ acts on the set of boundary divisors, preserving their types ($A,B,C,D$).

\end{definition}

\begin{remark}\label{rem:alldivisors}
Each boundary divisor corresponds to a family of $H$-invariant cycles whose irreducible components are contained in some fundamental $H$-invariant subvarieties. For instance, in the case of a divisor $A_i$, $i\neq 1$, following Section \ref{ssec:boundarydiv}, we consider the fixed component $Y_e \ni eB$ of the $\C^*$-action associated with $A_1$, we take a permutation $\sigma$ sending $1$ to $i$, and set $Y:=\sigma(Y_e)$. Note that $Y$ is independent of the choice of $\sigma$. Then $A_i$ parametrizes cycles having irreducible components contained in  $F^+_{A_i}(Y)$ and in $F^+_{B_i}(Y)$. In particular, if two permutations $\sigma,\sigma'$ satisfy that $\sigma(1)=\sigma'(1)$, then the divisors $A_{\sigma(1)}, A_{\sigma'(1)}\subset X$ are the same; a similar property holds for the divisors of types $B,C,D$. 
\end{remark}

\begin{notation}\label{not:alldivisors}
Abusing notation, the closed sets obtained as images of the boundary divisors $A_i,B_i,C_{ij},D_{ij}\subset X$ via a contraction $\pi_\sigma$ will be denoted again by $A_i$, $B_i$, $C_{ij}$, $D_{ij}\subset X_\sigma$, for every $\sigma\in W$. The same convention will apply to the images of these divisors under any contraction of  $X$. We will use this notation particularly often in the case of the images of the boundary divisors by the composition $p\circ \pi_e:X\to X_e \to \P^3$.
\end{notation}

\begin{remark}\label{rem:Dij=Dji}
Consider the permutation $r_2$, and the divisors $D_{12}\subset X_e$ and  $D_{21}=\phi_{r_2}(D_{12})\subset X_{r_2}$. Using the notation  presented in Section \ref{ssec:boundarydiv}, $D_{21}\subset X_e$ is determined by two invariant subsets $F_{C_{12}}^+(Y_e)$,  $F_{C_{03}}^+(Y_e)$; since $r_2$ preserves the partitions $C_{12}$ and $C_{03}$, the divisor $D_{21}$ is determined by the invariant sets $F_{C_{12}}^+(r_2(Y_e))$,  $F_{C_{03}}^+(r_2(Y_e))$. One may easily check that $r_2(Y_e)=Y_e$, hence the proper transforms of $D_{12}$ and $D_{21}$ in $X$ coincide. Then, by the action of $W$, we conclude that, in $X$,
$$
D_{ij}=D_{ji} \mbox{ for every } i\neq j. 
$$
In particular we have, at most, six boundary divisors of type $D$.
\end{remark}

\begin{remark}\label{rem:actFO}
By construction, the induced actions of $W$ 
on the sets: 
\[
\{A_i,\,\,0\leq i\leq 3\},\ \ \{B_i,\,\,0\leq i\leq 3\},\ \ \{C_{ij},\,\,0\leq i<j\leq 3\},\ \ \{D_{ij},\,\,0\leq i<j\leq 3\},
\] 
are compatible with the indexing and the identification of $W$ as the group of permutations of $\{0,1,2,3\}$. 
As for the action of $\tau$, we have \begin{equation}\label{eq:tauondiv}\begin{array}{l}A_i\stackrel{\tau}{\longleftrightarrow} B_{w_0(i)},\quad D_{ij}\stackrel{\tau}{\longleftrightarrow} D_{w_0(j) w_0(i)}\quad (w_0=(03)(12)\in W),\\[3pt] C_{12}\stackrel{\tau}{\longleftrightarrow}C_{03},\quad \tau(C_{ij})=C_{ij} \mbox{ for }(i,j)\neq (1,2), (0,3).\end{array}
\end{equation}
\end{remark}

Furthermore, using the coordinate expression of $r_1,r_2,r_3$ and $\tau$ given in Proposition \ref{prop:birmaps}, we may check that these divisors have different images via the contraction $p \circ\pi_e:X\to\P^3$ and, in particular, they are $20$ different prime divisors in $X$. The result of our computations, for which we have used {\tt SageMath}, are shown in the following statement, where we use the set of homogeneous coordinates $[x_0:x_1:x_2:x_3]$ in $\P^3$ introduced in Equation (\ref{eq:change}).

\begin{corollary}\label{cor:birmaps2}
The images via $p\circ \pi_e:X\to\P^3$ of the boundary divisors of $X$ 
are the subvarieties listed in Table \ref{tab:20}.
\begin{table}[h!]
\begingroup
\renewcommand*{\arraystretch}{1.1}
\begin{tabular}{|C|C|C||C|C|C|}
\hline
A_0 & \text{line} & x_1=x_3=0 & B_0 &\text{plane}  & x_1-x_3=0\\
\hline
A_1 & \text{line} &x_0=x_2=0 &B_1 &\text{plane}   &x_0-x_2=0 \\
\hline
A_2 & \text{plane} &x_0-x_1=0 & B_2 & \text{line}&x_0=x_1=0 \\
\hline
A_3 & \text{plane} &x_2-x_3=0 & B_3 & \text{line}&x_2=x_3=0 \\
\hline\hline
C_{01} &\text{point}&[1:1:1:1] & D_{01} &\text{line}&x_0-x_2=x_1-x_3=0\\
\hline
C_{02} &\text{plane}&x_0=0 & D_{02} &\text{point}&[0:0:1:0]\\
\hline
C_{03} &\text{plane}&x_2=0 & D_{03} &\text{point}&[1:0:0:0]\\
\hline
C_{12} &\text{plane}&x_1=0 & D_{12} &\text{point}&[0:0:0:1]\\
\hline
C_{13} &\text{plane}&x_3=0 & D_{13} &\text{point}&[0:1:0:0]\\
\hline
C_{23} &\text{quadric}& x_0x_3-x_1x_2=0 & D_{23} &\text{line}& x_0-x_1=x_2-x_3=0 \\
\hline
\end{tabular}\par\medskip
\caption{Images in $\PP^3$ of the boundary divisors.\label{tab:20}}
\vspace{-20pt}
\endgroup
\end{table}
\end{corollary}

\begin{proof}
Under the change of coordinates (\ref{eq:change}), we easily compute the equations of  $A_1$, $B_2$, $C_{02}$, $D_{12}$. Then,
applying the automorphisms $r_1,r_3$, we get the equations of $A_0$, $B_3$, $C_{03}$, $C_{12}$, $C_{13}$, $D_{02}$, $D_{03}$, $D_{13}$. 
We can then check that $r_2$ sends:
\begin{itemize}[leftmargin=\xx pt]
\item the plane  $x_0-x_1=0$ to $A_1$, 
\item the plane $x_0-x_2=0$ to $B_2$, 
\item the line $x_0-x_2=x_1-x_3=0$ to $D_{02}$, and
\item the quadric $x_0x_3-x_1x_2=0$ to $C_{13}$.
\end{itemize}
We then conclude that these varieties are $A_2,B_1,D_{01},C_{23}\subset\P^3$, respectively.

Finally, the equations of the remaining subvarieties are computed  applying $r_1$ to $B_1$, $r_2$ to $C_{02}$, $r_3$ to $A_2$, and $\tau$ to $D_{01}$.
\end{proof}

We have represented these subvarieties of $\P^3$ in Figure (\ref{fig:20}). The square represents the quadric $C_{23}$; the rest are either lines contained in $C_{23}$ ($A_0,A_1,B_2,B_3,D_{01}$, $D_{23}$), points in $C_{23}$($C_{01},D_{02},D_{03},D_{12},D_{13}$), or tangent spaces to $C_{23}$ at points ($A_2$, $A_3$, $B_0$, $B_1$, $C_{02}$, $C_{03}$, $C_{12}$, $C_{13}$, represented as the colored triangular shapes).

\begin{figure}[h!]
\includegraphics[width=6cm]{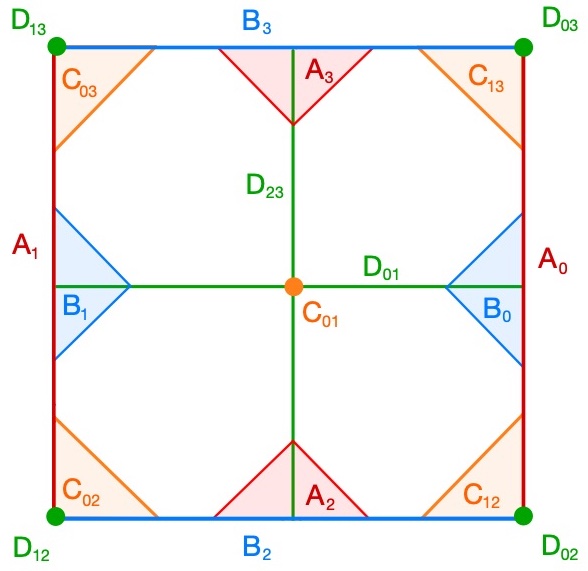}
\caption{Images in $\PP^3$ of the boundary divisors.\label{fig:20} }
\end{figure}

\section{The tile threefold}\label{sec:constructX}

In this section we construct a smooth birational modification $X'$ of $\P^3$, which we call the {\em tile threefold}, on which the birational maps induced by the elements of $\FO\subset\Bir(\P^3)$ are isomorphisms, and prove that it is isomorphic to the Chow quotient $X$ of $F$. In a nutshell, the tile threefold $X'$ is a resolution of indeterminacies of the birational maps of the tile group, compatible with the contractions $X_\sigma\to \P^3$;  we may then use the description of the Chow quotient $X$ as inverse limit of the $X_\sigma$'s to identify $X$ and $X'$.

\begin{definition}\label{def:X'}
The {\em tile $3$-fold} $X'$ is the smooth threefold obtained from $\P^3$ by blowing up first $A_0 \sqcup A_1$, then the strict transform of  $B_2 \sqcup B_3$,
then $C_{01}$, then the strict transform of $D_{01}, D_{23}$ and the inverse images of the remaining $D_{ij}$ -- the notation  here is the one of Table \ref{tab:20}.
We denote by $q:X' \to \PP^3$ the composition of these blowups. Abusing notation, we denote by $A_0,A_1,\dots,D_{23}\subset X'$ the prime divisors defined as proper transforms in $X'$ of the corresponding subsets of $\P^3$.
\end{definition}

Note that $X'$ is smooth, rational and $\rho_{X'}=12$. The goal of this section is to prove the following statement, that constitutes the first part of Theorem \ref{thm:main}:
\begin{proposition}\label{prop:XisoX'}
With the above notation, there exists an isomorphism:
\begin{equation}X'\simeq X. \label{eq:X'X}\end{equation} 
In particular, the 
Chow quotient $X$ of $F$ 
is a smooth rational threefold of Picard number $12$.
\end{proposition}

\begin{remark}\label{rem:ordering}
Note that we can make some changes in the order of the blowups described above, while still obtaining the same variety $X'$. Any ordering that respects the following rules is allowed:
\begin{itemize}[leftmargin=\xx pt]
\item the blowup of $D_{ij}$, $(i,j)\neq (0,1),(2,3)$ must be performed after the blowups of $A_i$ and $B_j$ (not necessarily in this order, see Remark \ref{rem:A1B2D12});
\item the blowup of $D_{01}$ (respectively $D_{23}$) must be performed after the blowup of $A_0,A_1$ (resp. $B_2,B_3$) and $C_{01}$.
\end{itemize} 
\end{remark}

By the universal property of the blowup, the isomorphisms $r_1,r_3,\tau\in \FO$ extend to automorphisms of $X'$, which we also denote by $r_1,r_3,\tau$:
\[
\xymatrix@R=18pt@C=35pt{X'\ar[r]_{q}\ar[d]_{r_1,r_3,\tau}&\P^3\ar[d]^{r_1,r_3,\tau}\\
X'\ar[r]_{q}&\P^3}
\]
In order to extend also the birational map $r_2$ to an automorphism of $X'$, we need the following factorization of $q$:

\begin{lemma}\label{lem:facto}
Let $p:X_e\to\P^3$ be the contraction defined in formula (\ref{eq:projP3}), and $p_{c}:Y\to X_e$ be the blowup of $X_e$ along the point $C_{01}$. Then there exists a contraction $p':X'\to Y$ such that we have a factorization:
\[
\xymatrix@C=35pt{X'\ar[r]^{p'}\ar@/_10pt/[rrr]_q&Y\ar[r]^{p_c}&X_e\ar[r]^p&\P^3}
\]
\end{lemma}

\begin{proof}
By Remark \ref{rem:ordering}, it is enough to define $p'$ as the composition of the sequence of blowups of $Y$ along the strict transform of $A_0$, then along  the strict transform of  $B_3$, then along the strict transforms of $D_{01}, D_{23}$, and finally along the inverse images of $D_{02},D_{03},D_{13}$.
\end{proof}

\begin{lemma}\label{lem:r2}
The birational automorphism $r_2$ of $\P^3$ 
extends to an automorphism of $X'$ (which, abusing notation, we denote again by $r_2$) so that we have a commutative diagram: 
\[
\xymatrix@R=18pt@C=30pt{X'\ar[r]^{p_c\circ p'}\ar[d]_{r_2}^{\simeq}&X_e\ar[r]^p\ar@{-->}[d]^{r_2}&\P^3\ar@{-->}[d]^{r_2}\\
X'\ar[r]^{p_c\circ p'}&X_e\ar[r]^p&\P^3}
\]
\end{lemma} 

\begin{proof}
Let us describe a resolution of the birational involution $r_2: \PP^3  \dasharrow \PP^3$. Its base locus scheme is the union of the reducible conic $x_0=x_1x_2=0$ (union of $A_1$ and $B_2$) and of the point $C_{01}$. Following Remark \ref{rem:A1B2D12} the blowup of $\P^3$ along the reducible conic is isomorphic to the toric variety $X(\Sigma_0)$, and the blowup of its singular point is the smooth threefold $X_e$. 
We thus get a commutative diagram:   
$$
 \xymatrix@R=18pt@C=35pt{X_e \ar@{-->}[d]_{r_2}  \ar[r]^{p} & \PP^3\ar@{-->}[d]^{r_2}\\ X_{e} \ar[r]^{p}  & \PP^3}
 $$
To complete the resolution we need to blowup the point $C_{01} \in X_e$, obtaining the variety $Y$ and an automorphism $r_2$ of $Y$ fitting in the commutative diagram:
 $$
 \xymatrix@R=18pt@C=35pt{Y \ar[r]^{p\circ p_c}  \ar[d]_{r_2}^{\simeq} & \PP^3 \ar@{-->}[d]^{r_2}\\ Y \ar[r]^{p\circ p_c}  & \PP^3}
 $$
Now we note that $r_2$ stabilizes the union of $A_0,B_3,D_{01},D_{02},D_{03},D_{13},D_{23}$ in $Y$; since this union is the center of $p':X'\to Y$ (see the proof of Lemma \ref{lem:facto} above), then  $r_2$  extends to an automorphism of $X'$.
\end{proof}
 
As a consequence, we get the following: 

\begin{corollary}\label{cor:FOautom}
There exists a monomorphism $\FO\to \Aut(X')$ such that the maps $p_c\circ p':X'\to X_e$, $p:X_e\to \P^3$ are $\FO$-equivariant.
\end{corollary}

For every $\sigma\in W\subset \FO$, we will denote by $\sigma:X'\to X'$ the corresponding automorphism of $X'$, and set $$\pi'_{\sigma}:=\phi_\sigma\circ (p_c\circ p')\circ \sigma^{-1}:X'\to X_\sigma,$$ so that 
for every $\sigma,w\in W$ we get a commutative diagram:
\begin{equation}\label{eq:phiw2}
\xymatrix@C=30pt{X'\ar[r]^{\sigma}\ar[d]_{\pi'_{w}}&X'\ar[d]^{\pi'_{\sigma w}}\\ X_w\ar[r]_{\phi_\sigma}&X_{\sigma w}}
\end{equation}
\begin{remark}\label{rem:XisoX'}
By construction, the exceptional locus of $\pi'_\sigma$ is the image via $\sigma$ of the exceptional locus of $p_c\circ p'$, i.e., the union of the divisors  $A_0,B_3,C_{01}, D_{ij}\neq D_{12} \subset X'$. The images of these divisors via $q:X'\to \P^3$ are, by construction, the linear subspaces $A_{\sigma(0)}$, $B_{\sigma(3)}$, $C_{\sigma(0)\sigma(1)}$, $D_{\sigma(i)\sigma(j)}\subset \P^3$, $(i,j)\neq(1,2)$.  
\end{remark}

We may now prove the main statement of the section:

\begin{proof}[Proof of Proposition \ref{prop:XisoX'}]
By Formula (\ref{eq:defsigma}), we have that $\pi_\sigma\circ(\pi_e)^{-1}$ is equal to $\phi_\sigma\circ\sigma^{-1}$ (as birational maps $X_e\dashrightarrow X_\sigma$) for every $\sigma \in W$. On the other hand, Corollary \ref{cor:FOautom} tells us that:
$$\phi_\sigma\circ\sigma^{-1}=\phi_\sigma\circ (p_c\circ p')\circ \sigma^{-1}\circ (p_c\circ p')^{-1},$$
where, on the right hand side, $\sigma^{-1}$ denotes the corresponding automorphism of $X'$. Finally, by the definition of $\pi'_e$ and $\pi'_\sigma$, we conclude that the above map equals $\pi'_\sigma\circ(\pi'_e)^{-1}$, that is, we have a 
commutative diagram:
$$
\xymatrix@C=40pt@R=6pt{&&X_e\ar@{-->}[dd]^{\phi_\sigma\circ\sigma^{-1}}\\
X'\ar@/_5pt/[drr]_{\pi'_\sigma}\ar@/^5pt/[urr]^{\pi'_e}&X\ar[dr]^{\pi_\sigma}\ar[ur]_{\pi_e}&\\
&&X_{\sigma}}
$$ 
By Proposition \ref{prop:invlim}, $X$ is isomorphic to the normalization $\ol{X}$ of its image into the product $\prod_{\sigma\in W}X_\sigma$, hence we have a morphism $\psi:X'\to X$, that commutes with the arrows of the diagram above.

The varieties $X,X'$ are normal and $\psi$ is, by construction, birational; thus, in order to prove that $\psi$ is an isomorphism, it is enough to show that it is finite. If this were not the case, there would be a curve $C\subset X'$ contracted by $\pi'_\sigma$, for every $\sigma \in W$. We conclude by observing, in light of Remark \ref{rem:XisoX'}, that the intersection of the exceptional loci of the morphisms $\pi'_\sigma$, $\sigma \in W$, is empty.
\end{proof}

\section{Anticanonical divisor of the tile threefold}\label{sec:antican}

In this section we study the boundary divisors in $X$, and describe their intersections. As a consequence, we show that $X$ is a weak Fano threefold, completing the proof of Theorem \ref{thm:main}.

\begin{notation}\label{not:birgeomX}
	In the sequel we will identify the Chow quotient $X$ with the tile $3$-fold $X'$, and decompose the birational map $q: X\to \P^3$ of Definition \ref{def:X'} as:
	 \[
	 \xymatrix@C=35pt{X\ar[r]^{q_d}\ar@/_10pt/[rrrr]_q&X_{c}\ar[r]^{q_c}&X_b\ar[r]^{q_b}&X_a\ar[r]^{q_a}&\P^3}
	 \]
	where $q_a,q_b,q_c$ and $q_d$ denote, respectively, the blowup of $\P^3$ along $A_0 \sqcup A_1$,  the blowup of $X_a$ along the strict transform of $B_2\sqcup B_3$,  the blowup of $X_b$ along $C_{01}$, and the blowup of $X_c$ along the strict transform of $D_{01}, D_{23}$ and the inverse images of the remaining $D_{ij}$.
\end{notation}

\subsection{Geometry of the boundary divisors}\label{ssec:boundarydivisors}

The set of boundary divisors contains all the exceptional divisors of $q$ and the strict transforms of some planes in $\P^3$, so it generates  $\Pic(X)$, which has rank $12$. 

We compute the pullback to $X$ of the ample generator $H$ of $\Pic(\PP^3)$ considering the total transform of the eight planes appearing in Table \ref{tab:20}, obtaining some effective divisors linearly equivalent to $q^*H$ (see Table \ref{tab:rels}). 

\begin{table}[h!!]
\begingroup
\renewcommand*{\arraystretch}{1.1}

\begin{tabular}{|C||C|}
\hline
\text{Plane} &q^*H\\
\hline\hline
A_2 & A_2+B_2+  C_{01} + D_{02}+D_{12}+D_{23}\\
\hline
A_3 & A_3+B_3+  C_{01}+D_{03} + D_{13}+D_{23}\\
\hline
B_0 & A_0+B_0+  C_{01}+D_{01} + D_{02}+D_{03}\\
\hline
B_1 &A_1+B_1+  C_{01}+D_{01} + D_{12}+D_{13}\\
\hline
C_{02} & A_1+B_2+  C_{02}+D_{02} + D_{12}+D_{13}\\
\hline
C_{12} & A_0+B_2+  C_{12}+D_{02} + D_{12}+D_{03}\\
\hline
C_{13} & A_0+B_3+  C_{13}+D_{02} + D_{03}+D_{13}\\
\hline
C_{03} & A_1+B_3+  C_{03} + D_{03}+D_{12}+D_{13}\\
\hline 
\end{tabular}\par\medskip
\caption{Eight linear equivalence relations in $X$.\label{tab:rels}}
\vspace{-20pt}
\endgroup
\end{table}

Moreover, considering the total transform of the quadric $C_{23}$, we obtain
\[
2q^*H=A_0+A_1+B_2+B_3 + C_{01} + C_{23} +D_{01} + D_{02}+ D_{03}+D_{12} + D_{13}+D_{23}.
\]

We checked that, using these expressions of $q^*H$ and $2q^*H$ we get $8$ independent relation among the $20$ boundary divisors on $X$. Since the Picard number of $X$ is twelve, and the boundary  divisors are generators of $\Pic(X)$, the Picard group is the quotient of the free abelian group generated by the boundary divisors by the relations described above.

We will now describe the boundary divisors and their mutual intersections. In view of the $\FO$-action (see Remark \ref{rem:actFO}), it is enough to describe one divisor of type $A$, one of type $C$, and one of type $D$. 

\medskip

\noindent{\bf (Description of $\mathbf {A_0}$).} 
After the blowup $q_a$, $A_0$ is isomorphic to $\PP^1 \times \PP^1$. Two fibers of $q_a$ map to $D_{02}$ and $D_{03}$, and another one to a point of $B_1$. The planes $B_0,C_{12}, C_{13}$  meet $A_0$ in minimal sections of $q_a|_{A_0}$, while the quadric $C_{23}$ intersects $A_0$ in a conic. 
We then have to blowup the points $B_2$, $B_3$ and $D_{01}$, obtaining a del Pezzo surface of degree five. The incidence graph of its $(-1)$-curves is the {\em Petersen graph}, that we have represented in Figure  \ref{fig:A01} (we have chosen an $S_3$-symmetric representation to highlight the action of the stabilizer of $A_0$ in $W$).

\begin{center}
\begin{figure}[h]
\includegraphics[width=12cm]{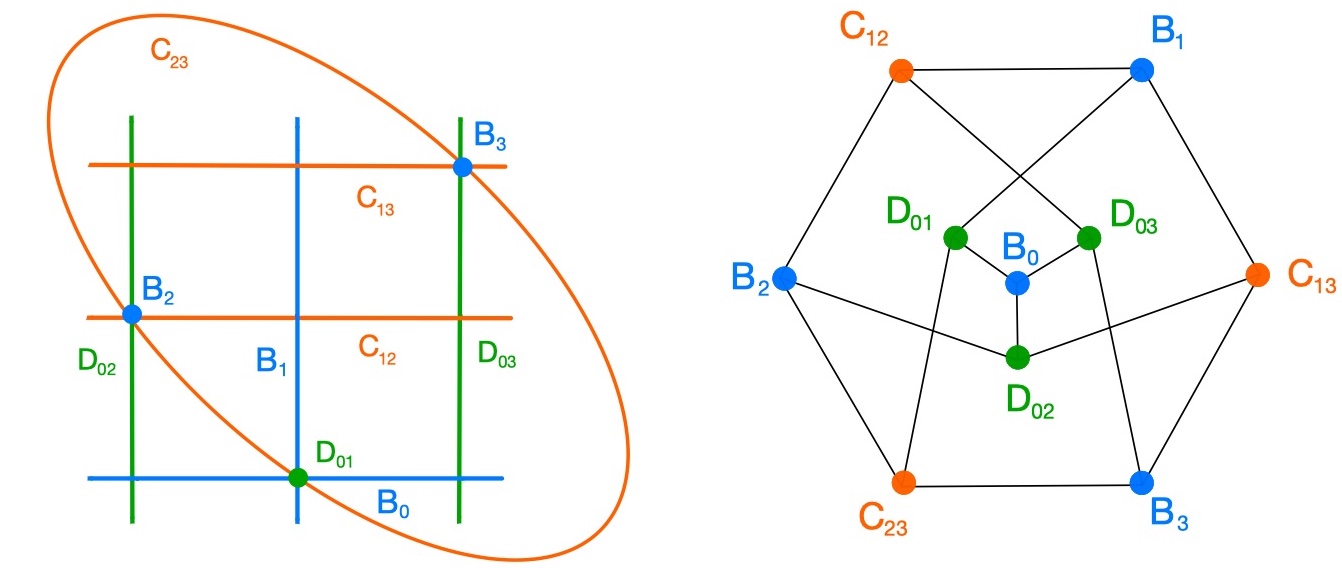} 
\caption{The divisor $A_0\subset X_a$ and the graph of $(-1)$-curves of the boundary divisor $A_0\subset X$.\label{fig:Petersen}}
\label{fig:A01}
\end{figure}
\end{center}

\noindent{\bf (Description of $\mathbf {C_{23}}$).}  
The only blowup which is not an isomorphism when restricted to $C_{23}$ is the blowup of the point $C_{01}$, hence $C_{23}$ is the blowup of $\PP^1 \times \PP^1$ at a point. Since $C_{23}$ contains the lines $A_0,A_1,B_2, B_3, D_{01}$ and $D_{23}$, the corresponding exceptional divisors meet $C_{23}$, while the other boundary divisors do not meet it. 

\medskip

\noindent{\bf (Description of $\mathbf {D_{01}}$).} 
This divisor is isomorphic to $\PP^1 \times \PP^1$. In $\PP^3$ the line $D_{01}$ meets $A_0, A_1$ and $C_{01}$ at a point, so the corresponding divisors meet $D_{01}$ in fibers of $q_d$. On the other hand 
the divisors $B_0,B_1$ and $C_{23}$ contain the line $D_{01}$, so their strict transforms meet $D_{01}$ in minimal sections of $q_d|_{D_{01}}$. We have represented $C_{23}$ and $D_{01}$  in Figure \ref{fig:CD}.
\begin{center}
\begin{figure}[h]
\includegraphics[width=11
cm]{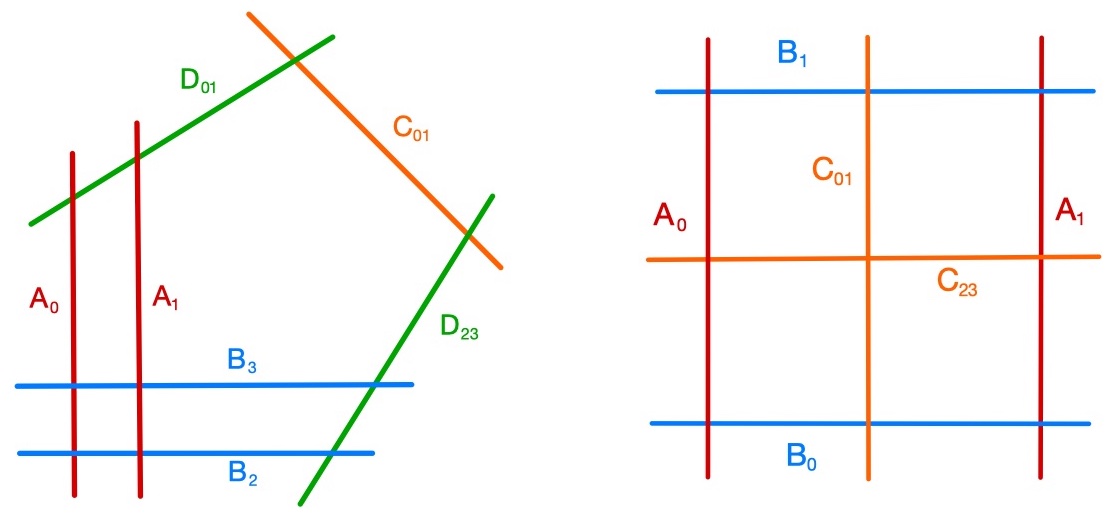}
\caption{The divisors $C_{23}$ and $D_{01}$.}
\label{fig:CD}
\end{figure}
\end{center}

\subsection{Intersection numbers}\label{ssec:intnumbers}
 The goal of this section is to use the descriptions  presented in Section \ref{ssec:boundarydivisors} to compute the intersection numbers of the form $E \cdot F \cdot G$, where $E$, $F$ and $G$ are boundary divisors. Taking into account the action of the group $\FO$ (described in Remark \ref{rem:actFO}), the information necessary to compute these numbers is contained in the following statement:

\begin{proposition}\label{prop:intnumbers}
The intersection numbers $E \cdot F \cdot G$ of triplets of boundary divisors is determined, up to the action of $\FO$, by the following rules:
\begin{itemize}[leftmargin=\xx pt]
\item[1.] Intersection numbers $E\cdot F\cdot C_{23}$, with $E,F\neq C_{23}$, are zero, with the following exceptions: 
\begin{center}
\setlength{\tabcolsep}{2pt}
\begin{tabular}{|C||C|C|C|C|C|C|C|C|C|C|C|C|C|C|C|}
\hline 
E& \,\,A_0\,\,  &  \,\,A_0\,\,  &  \,\,A_1\,\,  &  \,\,A_1\,\,  &  \,\,A_0\,\,  &  \,\,A_1\,\,  & \,\,B_2 \,\,& \,\,B_3\,\, & \,D_{01}\, & \,D_{23}\, &\,C_{01} \,& \,D_{01}\, & \,D_{23}\, \\
\hline
F&B_2 & B_3 & B_2 & B_3 & D_{01} &  D_{01} &  D_{23} &  D_{23} & C_{01} & C_{01} &C_{01} & D_{01} & D_{23} \\
\hline\hline
E\cdot F\cdot C_{23}&1&1&1&1&1&1&1&1&1& 1&-1&-1&-1\\
\hline
\end{tabular}
\end{center}
\medskip
\item[2.] Intersection numbers $E\cdot F\cdot D_{01}$, with $E,F\neq D_{01}$, are zero, with the 
exceptions:
\begin{center}
\setlength{\tabcolsep}{2pt}
\begin{tabular}{|C||C|C|C|C|C|C|C|C|C|}
\hline
E& \,\,A_0\,\,  & \,\,A_0\,\, &  \,\,A_1\,\,  &  \,\,A_1\,\,  &  \,\,A_0\,\,  &  \,\,A_1\,\,  &  \,\,B_0\,\, &  \,\,B_1\,\,  & \,C_{01}\,   \\
\hline
F&B_0 & B_1 & B_0 & B_1 & C_{23} &  C_{23} &  C_{01} &  C_{01} & C_{23}  \\
\hline\hline
E\cdot F\cdot D_{01} &1&1&1&1&1&1&1&1& 1\\
\hline
\end{tabular}
\end{center}
\medskip
\item[3.] Intersection numbers $E\cdot F\cdot A_0$, with $E,F\neq A_0$, are zero,  with the following exceptions:
\medskip
\begin{itemize}
\item[3.1.] $E^2\cdot A_0=-1$ if $E$ is a node of the Petersen graph in Figure \ref{fig:Petersen}, and
\item[3.2.] $E\cdot F\cdot A_0=1$ if $E,F$ are two connected nodes of the Petersen graph.
\end{itemize}
\medskip
\item[4.] $A_i^3=B_j^3=0$, $C_{ij}^3=1$, $D_{ij}^3=2$, for every $i,j$.
\end{itemize}

\end{proposition}

 \begin{proof}
 Rules 1, 2 and 3  follow from the description of the divisors $C_{23},D_{01},A_0$  presented in Section \ref{ssec:boundarydivisors}. By means of the action of $W$ we may compute all the intersection numbers  of boundary divisors
$E \cdot F \cdot G$, with $G=A_{\sigma(i)},C_{\sigma(i)\sigma(j)},D_{\sigma(i)\sigma(j)}$, in which at least two of the factors are different; furthermore, by applying $\tau$, we may also obtain the numbers in the case $G=B_{\sigma(j)}$.  
 
We are left with the case of the top self-intersections of boundary divisors. In this case we can use the linear equivalence relations obtained from Table \ref{tab:rels} to write every $E^3$ as a a cube-free formula. For instance,
\[A_1^3=(A_3+C_{12}-C_{23}-D_{01}+D_{03})\cdot A_1^2 = 0. 
\]
In this way we obtain the formulae listed in rule 4.
 \end{proof}

\subsection{The anticanonical divisor of $X$}\label{ssec:antican}

As a consequence of our study of the geometry of the boundary divisors in $X$ and their intersections, we will show in this section that the anticanonical divisor of $X$ is nef and big; in particular, $X$ is a Mori Dream Space. Moreover, $-K_X$ is globally generated, and its anticanonical model is a $3$-fold of degree $12$ in $\P^{13}$. 
Let us start with the following: 

\begin{lemma}\label{lem:nef} An effective divisor $D$ in $X$ which is a linear combination with nonnegative coefficients of the boundary divisors is nef if and only if its restriction to every boundary divisor is nef.
\end{lemma}

\begin{proof} If there exists an effective curve $\Gamma$ such that $D \cdot \Gamma <0$, then at least one of the boundary divisors is negative on $\Gamma$, hence it contains it.
\end{proof}

Using the description of $q$ as a sequence of blowups with smooth centers (Notation \ref{not:birgeomX}), we can write the anticanonical divisor as
\begin{equation}\label{eq:antican}
-K_X = 4 q^*H -(A_0+A_1+B_2+B_3 + D_{01} +D_{23}+2(D_{02}+D_{03}+D_{12}+D_{13}+C_{01})),
\end{equation}
which, writing $4 q^*H$ as the total transform of $A_2+A_3+B_0+B_1$ and using the linear relations presented in Table \ref{tab:rels}, 
can be rewritten as:
\begin{equation}\label{eq:can1}
-K_X =  A_2+A_3+B_0+B_1 + 2C_{01} +D_{01} +D_{23}.
\end{equation}
Moreover, if we write $4 q^*H$ as the total transform of $A_2+A_3+C_{23}$ we get: 
\begin{equation}\label{eq:can2}
-K_X =  A_2 + A_3 + B_2 + B_3 + C_{01} + C_{23} + 2D_{23}.
\end{equation}

We are now ready to prove the following statement, that constitutes the second part of Theorem \ref{thm:main}:

\begin{proposition}\label{prop:MDS} The Chow quotient $X$ is weak Fano.  
In particular it is a Mori Dream Space.
\end{proposition}

\begin{proof}
We recall first that smooth weak Fano varieties are known to be MDS's (cf. \cite[Corollary 1.3.2]{BCHM}, see also \cite[Example~3.2]{Castravet}), and that, by definition, $X$ is weak Fano if $-K_X$ is nef and big. 

In view of Lemma \ref{lem:nef}, to prove the nefness of $-K_X$ it is enough to check that $-K_X$ restricted to the boundary divisors is nef. We will use the expression of $-K_X$ given in formula (\ref{eq:can1}). Taking into account the action of $\FO$, it is enough to check nefness on one divisor of type $A$, one of type $C$ and one of type $D$; we will consider $A_0, C_{12}$, and $D_{12}$
 
As observed in Section \ref{ssec:boundarydiv} $A_0$ does not meet $A_2, A_3, C_{01}$ and $D_{23}$, hence
\[
-K_X|_{A_{0}} = (B_0+B_1+D_{01})|_{A_0}, 
\]
and this divisor is nef, because it is effective and has nonnegative intersection with each of the irreducible components of its support.  
Analogously we have:
\[
-K_X|_{C_{12}} = (A_3+B_1)|_{C_{12}}. 
\]
which is nef because $A_3|_{C_{12}}$ and $B_1|_{C_{12}}$ are irreducible, effective and not exceptional in $C_{12}$. 
Finally, we have: 
\[
-K_X|_{D_{12}} = (A_2+B_1)|_{D_{12}}, 
\]
and we conclude as in case of $C_{12}$.

Finally, since $-K_X$ is nef, to check its bigness it  is enough to check that $(-K_X)^3>0$. We have computed this number with  {\tt SageMath}, starting from formula (\ref{eq:can1}), and using the intersection numbers formulae of Proposition \ref{prop:intnumbers}:
 \[
 (-K_X)^3 =12.\qedhere
 \]
\end{proof}

\begin{proposition}\label{prop:anticansys} The anticanonical divisor of $X$ is globally generated, and the anticanonical model of $X$ is a  threefold of degree $12$ in $\P^{8}$.
\end{proposition}

\begin{proof}[Proof of Proposition \ref{prop:anticansys}]
Thanks to expression (\ref{eq:antican}), we can identify the linear system $|-K_X|$ with a subsystem of quartics in $\P^3$, namely those passing through the $6$ lines $A_0$, $A_1$, $B_2$, $B_3$, $D_{01},D_{23}\subset \P^3$, and singular at the points $D_{02}$, $D_{03}$ $D_{12}$, $D_{13}$, $C_{01}$ (each one contained in two of the six lines).  With the help of {\tt SageMath}, one can show that such linear system has dimension $9$. 
\[\begin{array}{l}
%
%
(x_1 x_2 - x_0 x_3)^2,\quad
 x_0   x_1   (x_1 - x_3)   (x_2 - x_3),\quad
 x_0   x_2   (x_1 - x_3)   (x_2 - x_3),\\
  x_0   x_1   (x_0 - x_2)   (x_2 - x_3),\quad
 x_0   x_3   (x_0 - x_2)   (x_2 - x_3),\quad
 x_1   x_3   (x_0 - x_2)   (x_2 - x_3),\\
 x_2   x_3   (x_0 - x_1)   (x_1 - x_3),\quad
 x_2   x_3   (x_0 - x_1)   (x_0 - x_2),\quad
 x_1   x_3   (x_0 - x_1)   (x_0 - x_2).

\end{array}
\]

A simple computation shows that, in the subset of $\P^3$ where all above generators are different from zero, the map defined by the linear system is injective; so, in particular the rational map defined by $|-K_X|$ is birational. Let us now show that $|-K_X|$ is also globally generated.

From the above list of generators we immediately see that there are no fixed points outside of the six lines $A_0$, $A_1$, $B_2$, $B_3$, $D_{01},D_{23}\subset \P^3$. Let us denote by $\cI$ the ideal sheaf of $\cO_{\P^3}$ generated by the above polynomials. By the universal property of the blowup, in order to see that the induced rational map from $X$ to $\P^8$ is a morphism it is enough to check that $\cI\otimes\cO_X$ is an invertible sheaf. This fact may be checked locally, at every point of the six lines.

We note now the following facts, that one may check by direct computation:
\begin{itemize}
\item the localization of $\cI$ at a point contained in precisely one of the six lines, is equal to the ideal of the line;
\item the localization of $\cI$ at one of the five singular points is the product of the ideals of the two lines containing it; for instance, the localization of $\cI$ at $D_{03}$ is $(x_1x_2,x_1x_3,x_2x_3,x_3^2)=(x_1,x_3)(x_2,x_3)$. 
\end{itemize}

Since the morphism $p\circ\pi_e:X\to \P^3$ factors by the blowup of $\P^3$ along $A_0$, and also by the blowup of $\P^3$ along $B_3$ (whose ideals are, respectively $(x_1,x_3)$, $(x_2,x_3)$), it follows that $\cI\otimes\cO_X$ is invertible on the inverse image of a neighborhood of $D_{03}$. A similar argument applies to the points $D_{02}$, $D_{12}$, $D_{13}$. Finally, in the case of the point $C_{01}$, we conclude by noting that the morphism $p\circ\pi_e$ factors through the successive blowup of $\P^3$ along the point $C_{01}$ and then the strict transform of the line $D_{01}$, and also by the successive blowup of $\P^3$ along the point $C_{01}$ and then the strict transform of the line $D_{23}$; since the inverse image of the ideal of the two lines by each of these two maps is an invertible sheaf, the conclusion follows.

Summing up, the linear system $|-K_X|$ provides a birational morphism $X\to \P^{8}$ whose image is a $3$-fold of degree $(-K_X)^3=12$. 
\end{proof}

\begin{remark}\label{rem:anticansys}
In the next section we will see that $X$ contains $19$ curves of $-K_X$-degree zero, that generate extremal rays of $\cNE{X}$ whose associated contractions are small. Hence the anticanonical model of $X$ has $19$ non-$\QQ$-factorial singularities. 
\end{remark}

\section{Birational geometry of the tile threefold}\label{sec:birgeomX}

We finish the paper by studying birational properties of $X$; in particular, we describe its Mori, nef and effective cones, $\cNE{X}$, $\Nef(X)$, $\Eff(X)$. Moreover, we portray the 
Chow quotients of the partial flags varieties of  $\C^4$ as contractions of $X$.

\subsection{The cone of curves of $X$}\label{ssec:moricone}

Let us start with the following:  

\begin{lemma}\label{lem:extremal1} Let  $\Gamma \subset X$ be an irreducible curve such that $-K_X\cdot \Gamma\in\{0,1\}$. 
Then $\Gamma$ has negative intersection with a boundary divisor. 
\end{lemma}

\begin{proof} First of all, since the boundary divisors generate the Picard group, they cannot all have intersection number zero with $\Gamma$. To prove the Lemma we will use some expressions of $-K_X$ as positive linear combination of boundary divisors. Starting from formula (\ref{eq:can1}) and applying the group action we get:
\begin{equation}\label{eq:multican1}
\begin{split} - K_X & =  A_2 + A_3 + B_0 + B_1 + 2C_{01} + D_{01} + D_{23}\\
& =A_1 + A_3 + B_0 + B_2 + 2C_{02} + D_{02} + D_{13}\\
 & =A_0 + A_1 + B_2 + B_3 + 2C_{23} + D_{01} + D_{23}\\
 & =A_1 + A_2 + B_0 + B_3 + 2C_{03} + D_{03} + D_{12}\\
 & =A_0 + A_3 + B_1 + B_2 + 2C_{12} + D_{03} + D_{12}\\
& = A_0 + A_2 + B_1 + B_3 + 2C_{13} + D_{02} + D_{13}.
\end{split}
\end{equation}
Doing the same with formula (\ref{eq:can2}) we get:
\begin{equation}\label{eq:multican2}
\begin{split} - K_X &  =A_2 + A_3 + B_2 + B_3 + C_{01} + C_{23} + 2D_{23} \\
&=  A_1 + A_3 + B_1 + B_3 + C_{02} + C_{13} + 2D_{13} \\
& = A_0 + A_2 + B_0 + B_2 + C_{02} + C_{13} + 2D_{02} \\
& =A_0 + A_3 + B_0 + B_3 + C_{03} + C_{12} + 2D_{03} \\
& =A_0 + A_1 + B_0 + B_1 + C_{01} + C_{23} + 2D_{01} \\
& =A_1 + A_2 + B_1 + B_2 + C_{03} + C_{12} + 2D_{12}.
\end{split}
\end{equation}

Note that every boundary divisors appear in at least one of the above expressions of $-K_X$, so if $-K_X\cdot \Gamma=0$, the conclusion is obvious. 

In the case $-K_X \cdot \Gamma =1$, assume for the sake of contradiction that no boundary divisor is negative on $\Gamma$. Then, noting that some coefficients are equal to two in the above expressions, we conclude that $C_{ij}\cdot \Gamma=D_{ij}\cdot\Gamma=0$, for every $i,j$. Since every possible pair $(A_i,A_j)$ (resp. $(B_i,B_j)$), $i\neq j$, appears in at least one expression of $-K_X$, we must have that precisely one $A_i$ and one $B_j$ have non zero intersection number with $\Gamma$. Then we get to a contradiction by noting that every pair $(A_i,B_j)$ appears in at least one of the above expressions.
\end{proof}

\begin{corollary}\label{cor:extremal2} Every extremal ray of $\cNE{X}$ contains the class of a curve contained in a boundary divisor. 
\end{corollary} 

\begin{proof}
Let $R$ be an extremal ray of $\cNE{X}$ and $\Gamma$ an irreducible curve whose class is a minimal generator of $R$. By Lemma \ref{lem:extremal1} we may restrict ourselves to the case in which $-K_X \cdot \Gamma \ge 2$. 

If the contraction $\varphi_R$ associated with $R$ has a fiber $F$ of dimension two, the boundary divisors which are positive on $\Gamma$ intersect $F$ along a curve, which is numerically proportional to $\Gamma$, and the statement holds.

We can thus assume that all the fibers of the contraction $\varphi_R$ are one dimensional. Then, by the fiber locus inequality (see \cite[Theorem~1.1]{Wis}), $\varphi_R$ is a contraction of fiber type. By the assumption on the minimality of $\Gamma$ in $R$, the contraction $\varphi_R$ is a $\PP^1$-bundle. We conclude observing that, in this situation, every boundary divisor whose intersection number with $\Gamma$ is zero contains fibers of $\varphi_R$.
\end{proof}

\begin{remark}\label{rem:SQMs}
In the proofs of Lemma \ref{lem:extremal1} and Corollary \ref{cor:extremal2} we have used only the properties of the anticanonical bundle $-K_X$ and of the Picard group of $X$, hence the statements are true for every small $\Q$-factorial modification of $X$.
\end{remark}

By Corollary \ref{cor:extremal2}  $\cNE{X}$ is generated by the classes of curves contained in boundary divisors, hence by the classes of curves which are extremal in the cones of curves of the boundary divisors, which are are known from Section \ref{ssec:boundarydivisors}. We have computed with {\tt SageMath} the classes of these curves, and checked that the cone they generate (which is $\cNE{X}$) has $31$ extremal rays ($12$ $K_X$-negative, $19$ $K_X$-trivial), each of them generated by a curve that is the intersection of two boundary divisors. 
We will denote a curve which is the intersection of the boundary divisors $E$ and $F$ by $E\tast F$, in agreement with the notation used in the {\tt SageMath} files.

The $12$ $K_X$-negative rays are the $\FO$-orbit of the ray generated by $A_0\tast D_{01}$. The exceptional locus of each of them is a divisor $D_{ij}$ (each $D_{ij}$ can be contracted in two different ways).
The $19$ $K_X$-trivial extremal rays are small; their exceptional locus is a smooth rational curve which is the intersection of two boundary divisors which contain it as a $(-1)$-curve. In particular the normal bundle of the curve in $X$ is $\cO(-1) \oplus \cO(-1)$, and the flop of the curve is an Atiyah flop. Those curves belong to three different orbits, namely, $12$ of them are the orbit of $A_0 \tast B_1$, $4$ of them are the orbit of $A_0 \tast B_0$ and the last $3$ are the orbit of $C_{01}\tast C_{23}$. The following statement summarizes this information:

\begin{proposition}\label{prop:moricone}
The extremal rays of the Mori cone $\cNE{X}$ of the tile $3$-fold $X$ 
are generated by the classes of $31$ curves, listed below: 
 \end{proposition}
\begin{center}
\renewcommand*{\arraystretch}{1.1}
 \begin{tabular}{|c||c|c|}
 \hline
 Type&Exceptional Locus&Ray generators\\\hline\hline
 \multirow{ 2}{*}{Divisorial}&\multirow{2}{*}{$D_{ij}$, $i\neq j$}&$[A_i\tast D_{ij}]=[A_j\tast D_{ij}]$\\\cline{3-3}
 &&$[B_i\tast D_{ij}]=[B_j\tast D_{ij}]$\\\hline
 \multirow{ 3}{*}{Small}&$A_i \tast B_j$, $i \neq j$&$[A_i \tast B_j]$\\\cline{2-3}
 &$A_i \tast B_i$&$[A_i \tast B_i]$\\\cline{2-3}
 &$C_{\sigma(0)\sigma(1)}\tast C_{\sigma(2)\sigma(3)}$, $\sigma\in W$&$[C_{\sigma(0)\sigma(1)}\tast C_{\sigma(2)\sigma(3)}]$\\\hline
 \end{tabular}
 \end{center}\par\medskip

 \begin{remark}\label{rem:fvec}
	Knowing the intersection numbers of these curves with every boundary divisor, we can choose basis of $\NU(X)$ and $\Nu(X)$, find coordinates of the rays, and compute the whole combinatorial structure of $\cNE{X}$ with {\tt SageMath}. The complete $f$-vector of $\cNE{X}$, containing the number of faces of dimension $i$ in the position $i\leq 11$ is the following:
	\begin{equation*}\label{eq:f-vector}
		\big(\,31,\,\, 387,\,\, 2647,\,\, 10942,\,\, 28495,\,\, 47531,\,\, 50616,\,\, 33484,\,\, 12912,\,\, 2544,\,\, 189\,\big).
	\end{equation*}
 \end{remark}

\subsection{The nef cone of $X$}\label{ssec:nefcone}

One may also compute with {\tt SageMath} minimal generators of the $189$ extremal rays of $\Nef(X)$ (corresponding to facets of $\cNE{X}$, see Remark \ref{rem:fvec}), which are supporting divisors of the contractions of $X$ to varieties with Picard number one. 
By means of Proposition \ref{prop:intnumbers}, we may compute the top self-intersection of these generators. 
We obtain that, among the $189$ maximal contractions of $X$, $20$ are of fiber type, and $169$ are birational contractions, with eight possible values of the top self intersection of a minimal supporting divisor $L$: \par  
\medskip
\begin{center}
\begin{tabular}{|l||c|c|c|c|c|c|c|c|c|}
\hline
$L^3$ & $0$ & $1$ & $2$ & $4$ & $5$ & $14$ & $16$ & $18$ & $22$ \\\hline
No. of contractions& $20$ & $6$ & $24$ & $6$ & $48$ & $6$ & $15$ & $16$ & $48$ \\\hline
\end{tabular}
\end{center}\par\medskip

Let us now study in finer detail those $20$ maximal fiber type contractions of $X$. Using again Proposition \ref{prop:intnumbers}, we see that $9$ of these contractions are supported by divisors whose squares are numerically trivial.
Equivalently the targets of these contractions are curves and, since they are necessarily rational, then they are projective lines. 

Among the supporting divisors of these contractions, there is a $\FO$-orbit containing $8$ elements, consisting of the $W$-orbits of:
$$
B_0+C_{01}+D_{01},\qquad A_0+C_{23}+D_{01}.
$$
Thus we may  conclude that the remaining contraction to $\P^1$, which is supported by the divisor 
\begin{equation}\label{eq:ChowGrass}
C_{01} + C_{23} + D_{01} + D_{23},
\end{equation}
is $\FO$-invariant. Summing up:

\begin{proposition}\label{prop:9toP1}
The variety $X$ has exactly $9$ contractions to $\P^1$, supported by the divisors described above. They are classified into two $\FO$-orbits, containing, respectively, $8$ and $1$ contractions.
\end{proposition}

Examining the remaining $11$ fiber type contractions (whose image is a surface), we checked that there are three $\FO$-orbits, containing, respectively, $1$, $2$ and $8$ elements. The following divisors support one contraction of each class:
\begin{itemize}[leftmargin=\yy pt]
\item[(1)] $A_0 + B_0 + D_{01} + D_{02} + D_{03}$;
\item[(2)] $A_2 + A_3 + C_{01} + D_{23}$;
\item[(8)] $B_0 + C_{01} + C_{02} + D_{01} + D_{02}$.
\end{itemize}

\begin{proposition}\label{prop:surfaces1}
The variety $X$ has exactly $11$ maximal contractions to surfaces, and their images are isomorphic to $\P^2$.
\end{proposition}

\begin{proof}
It is enough to show that the images of the contractions defined by the three divisors in the above list are $\P^2$. In order to do so, we first note that, by means of the relations described in Table \ref{tab:rels}, the  three divisors above can be written, respectively, as:
 \begin{itemize}[leftmargin=\yy pt]
\item[(1)] $q^*H-C_{01}$;
\item[(2)] $2q^*H-(B_2+B_3+C_{01}+D_{02} + D_{03} + D_{12} + D_{13} + D_{23})$; 
\item[(8)] $2q^*H-(A_0+A_1 + B_2 + D_{02} + D_{03} + D_{12} + D_{13})$. 
\end{itemize}
This allows us to identify the complete linear systems of these divisors with linear systems of hyperplanes and quadrics in $\P^3$, namely a linear system of hyperplanes passing by a point (case (1)), and two linear system of quadrics containing a chain of three lines (cases (2) and (8)). A straightforward computation shows that these linear systems are $3$-dimensional, so they provide a contraction from $X$ to $\P^2$. 
\end{proof}

\subsection{$W$-invariant contractions}\label{ssec:Winvcontr}

As a by-product of our description of $X$, we can obtain geometric descriptions of the Chow quotients of all the rational homogeneous $\PGL(4)$-varieties. Since the Chow quotient of the projective space $\P^3$ by the action of $H$ is a point, we are left with the following homogeneous varieties:

\begin{itemize}[leftmargin=22 pt]
\item[(i)] $F(2)$:  Grassmannian of $2$-dimensional subspaces of $\C^4$;
\item[(ii)] $F(1,2)$: variety of flags of subspaces of $\C^4$ of dimensions $1,2$;
\item[(iii)] $F(2,3)$: variety of flags of subspaces of $\C^4$ of dimensions $2,3$;
\item[(iv)] $F(1,3)$: variety of flags of subspaces of $\C^4$ of dimensions $1,3$.
\end{itemize}

The corresponding normalized Chow quotients by the actions of $H\subset\PGL(4)$ will be denoted by $X(2)$, $X(1,2)$, $X(2,3)$, $X(1,3)$. By definition, they have dimensions $1,2,2,2$, respectively. In particular $X(2)\simeq\P^1$. \par\medskip

The key observation here is the following statement:

\begin{lemma}\label{lem:partChow}
The natural morphisms between rational homogeneous varieties
$$
\xymatrix@=15pt{&&F(1,2)\ar[rd]&\\F(1,3)&F\ar[ru]\ar[rd]\ar[l]&&F(2) \\&&F(2,3)\ar[ru]&
}
$$
induce $W$-invariant contractions:
$$
\xymatrix@=15pt{&&X(1,2)\ar[rd]&\\X(1,3)&X\ar[ru]\ar[rd]\ar[l]&&X(2) \\&&X(2,3)\ar[ru]&
}
$$
\end{lemma}

\begin{proof}
We will write the proof for the case of the contraction $F\to F(1,2)$; the remaining cases are analogous.

Denoting by $u:Z\to X$ and $u':Z(1,2)\to X(1,2)$ the corresponding normalizations of the universal families of cycles in $F$ and $F(1,2)$, we get a commutative diagram:
$$
\xymatrix{X\ar[d]&Z\ar[r]^v\ar[l]_u\ar[d]&F\ar[d]\\X(1,2)&Z(1,2)\ar[r]^{v'}\ar[l]_{u'}&F(1,2)}
$$
Since the fibers of $F \to F(1,2)$ are connected and $v$ is birational, it follows that the composition $Z \to F(1,2)$ has connected fibers. As $v'$ is also birational, the induced map $Z \to Z(1,2)$ has connected fibers. Moreover, $u'$ has connected fibers, so the composition $Z \to X(1,2)$ has connected fibers. Consequently, the surjective map $X \to X(1,2)$ also has connected fibers.

Finally, the fact that this map is $W$-invariant follows from the fact that it is induced by the contraction $F\to F(1,2)$, which is $G$-equivariant.
\end{proof}

\begin{remark}\label{rem:partChow}
Note that the above proof gives an analogous statement for the Chow quotient of a rational homogeneous $G$-variety by the action of a maximal torus in $G$, for any semisimple algebraic group $G$.
\end{remark}

Our goal now will be to identify the faces of $\Nef(X)$ providing the  contractions from $X$ to the Chow quotients of the varieties of partial flags.

We start by observing that, with the above notation, the contractions to $X(2)$ and to $X(1,3)$ must be invariant by the action of $\FO$, whereas the the anti-trasposition $\tau$ exchanges  the contractions to $X(1,2)$ and to $X(2,3)$. This immediately implies  the following:

\begin{lemma}\label{lem:ChowGrass}
The contraction $X\to X(2)$ is supported by the divisor:
$$
L_1:=C_{01} + C_{23} + D_{01} + D_{23}
$$
\end{lemma}

\begin{proof}
It follows from the fact, up to scaling, this is the only $\FO$-invariant divisor supporting a contraction to $\P^1$.
\end{proof} 

We may now identify the Chow quotients of the remaining three varieties of partial flags:

\begin{proposition}\label{prop:ChowPFlag}
The Chow quotients $X(1,2)$ and $X(2,3)$ are isomorphic to $\P^2$ blown up in four points, and the Chow quotient of $X(1,3)$ is isomorphic to $\P^2$.
\end{proposition}

\begin{proof}
We note first that the contraction $X\to X(2)\simeq \P^1$ (supported by the divisor $L_1$ of Lemma \ref{lem:ChowGrass}) factors via $X(1,2)$, $X(2,3)$, but not via $X(1,3)$. On the other hand, we can check that the cone generated by the rays of the eight $\FO$-equivalent contractions to $\PP^1$ contains a big divisor. This implies that $X(1,2)$, $X(2,3)$ have either $1$ or $5$ contractions to $\P^1$, and that $X(1,3)$ has either zero or $4$ contractions to $\P^1$.

Let us deal first with the cases $X(1,2)$, $X(2,3)$. Arguing as in Section \ref{sec:nilpotent}, one may easily check that the combinatorial quotients of these varieties are blowups of $\P^2$ at two points, and so both $X(1,2)$, $X(2,3)$ have at least two contractions  to $\P^1$, hence, by the above argument, they have five contractions to $\PP^1$. In particular, the nef cones of these varieties, as subcones of $\Nef(X)$, contain one of the following two cones:
\[\begin{array}{l}
M_1=\langle L_1,A_0 + C_{23} + D_{01},A_1 + C_{23} + D_{01},A_2 + C_{01} + D_{23},A_3 + C_{01} + D_{23}
  \rangle\\
M'_1=\langle L_1,B_0 + C_{01} + D_{01},B_1 + C_{01} + D_{01},B_2 + C_{23} + D_{23},B_3 + C_{23} + D_{23}  \rangle  
\end{array} 
\] 
Obviously, we have that $\tau(M_1)=M'_1$. The dimension of these two cones is equal to $5$, and one may compute the intersection of their linear spans with $\Nef(X)$:
$$N_1:=\R(M_1)\cap \Nef(X),\quad N'_1:=\R(M'_1)\cap \Nef(X).$$
It turns out that, in each case, the intersection is the cone over a {\em joined pentachoron} (defined as the convex hull of the integer generators of its extremal rays), having $5$ more extremal rays, corresponding to contractions to $\P^2$, respectively:
\[\begin{array}{l} 
 A_2 + A_3 + C_{01} + D_{23},
 A_0 + L_1,
 A_1 + L_1,
 A_2 + L_1,
 A_3 + L_1, \quad\mbox{and}\\
B_0 + B_1 + C_{01} + D_{01}, B_0 + L_1, 
 B_1 + L_1,
 B_2 + L_1,
 B_3+L_1,
\end{array} 
\]
Furthermore, one can check that  the barycenters of these two joined pentachorons have top self-intersection $0$, and that any $W$-invariant face of $\Nef(X)$ containing properly $N_1$ or $N'_1$ contains a divisor of positive top self-intersection. We conclude that these are the nef cones of $X(1,2)$ and $X(2,3)$. 

We note now that the cone over the joined pentachoron, that has $10$ extremal rays (corresponding to $5$ contractions to $\P^1$ and $5$ contractions to $\P^2$), and $10$ facets (which are cones over  triangular tegums), is combinatorically equivalent to the nef cone of the del Pezzo surface of degree $5$. In order to check that $N_1$ and $N'_1$, as subcones of $\Nef(X)$, correspond indeed to contractions of $X$ isomorphic to a del Pezzo surface of degree $5$, we denote by $L_2,L_2'$ the generators of the rays corresponding to the barycenters of the two joined pentachorons:
$$
L_2=A_2 + A_3 + 2C_{01} + C_{23} + D_{01} + 2D_{23},\quad L'_2=B_0+B_1+2C_{01} + C_{23} + 2D_{01} + D_{23}.
$$
By a straightforward intersection computation, one can check that the restrictions of $L_2$ to the divisors $B_i$ are ample, and the restrictions of $L_2'$  to the divisors $A_i$ are ample; following Section \ref{ssec:boundarydivisors} these divisors are del Pezzo surfaces of degree $5$. This concludes the proof for $X(1,2)$ and $X(2,3)$.

In order to identify $X(1,3)$ we claim first that all its maximal contractions have $\P^2$ as target. In fact, as noted above, if $X(1,3)$ had contractions to $\P^1$, there would be $4$ of them, and the cone generated by the associated rays would be contained in the nef cone of either $X(1,2)$ or $X(2,3)$. But every facet of the nef cone of the del Pezzo surface of degree $5$ contains only $3$ rays corresponding to contractions to $\P^1$, hence the nef cone of $X(1,3)$ would contain a divisor in the interior of $\Nef(X(1,2))$ or  $\Nef(X(2,3))$, supporting a contraction to either $X(1,2)$ or $X(2,3)$, a contradiction. 

Now, if the nef cone of $X(1,3)$ contained more than one contraction to $\P^2$, arguing as above one could check that the barycenter of the two generators of the corresponding rays would either have positive top self-intersection (contradicting that $X(1,3)$ is a surface), or be ample on $X(1,2)$, $X(2,3)$, or one of their contractions of degree $6$, contradicting that $X(1,3)$ has no contractions to $\P^1$. 

We then conclude that the nef cone of $X(1,3)$ is the extremal ray generated by the $\FO$-invariant divisor $A_0 + B_0 + D_{01} + D_{02} + D_{03}$ corresponding to a contraction of $X$ onto $\P^2$.
\end{proof}

\subsection{The effective cone of $X$}\label{ssec:effcone}

In order to describe the effective cone of $X$ we need to introduce some new effective divisors. 
The planes in $\P^3$ given by: 
\begin{align*}
H_{\{01\}\{23\}}: \quad & x_0-x_1-x_2+x_3=0\\
H_{\{02\}\{13\}}: \quad & x_1-x_2=0\\
H_{\{03\}\{12\}}: \quad & x_0-x_3=0
\end{align*}
are fixed by the anti-transposition $\tau$, and their set is a $W$-orbit; the notation is chosen so that $$\{ij\}\{kl\}=\{ji\}\{kl\}=\{ij\}\{lk\}=\{kl\}\{ij\},$$ and one can check that $\sigma(H_{\{ij\}\{kl\}})=H_{\{\sigma(i)\sigma(j)\}\{\sigma(k)\sigma(l)\}}$ for every $\sigma\in W$. 
Note that $H_{\{01\}\{23\}}$ intersects the quadric $x_0x_3-x_1x_2=0$ in the lines $D_{01}$, $D_{23}$.
Its strict transform in $X$, denoted by $H_{\{01\}\{23\}}\subset X$ is linearly equivalent to
\[
 q^*H - C_{01} -D_{01}-D_{23},
\]
that is, to
\[
A_2 + B_2 - D_{01} + D_{02} + D_{12}.
\]
Using the group action we can find similar expressions in $\Pic(X)$ for the other two strict transforms:
\begin{align*}
H_{\{02\}\{13\}}&=A_1 + B_1 + D_{01} - D_{02} + D_{12},\\
H_{\{03\}\{12\}}&=A_0 + B_0 + D_{01} + D_{02} - D_{12}.
\end{align*}

We will need to consider also the Cayley's cubic surface in $\PP^3$, of equation:
\[
x_0x_1x_2 - x_0x_1x_3 - x_0x_2x_3 + x_1x_2x_3 =0.
\]
This surface, which is $\FO$-invariant, intersects the quadric $x_0x_3-x_1x_2=0$ in the six lines $A_0$, $A_1$, $B_2$, $B_3$, $D_{01}$, $D_{23}$ and has four nodes at the points $D_{02},D_{03},D_{12},D_{13}$. 
The linear equivalence class of its strict transform in $X$ is
\[
 S = 3q^*H -(A_0 + A_1 + C_{01} + B_2 + B_3 + D_{01} + D_{23} + 2(D_{02} + D_{03} + D_{13} + D_{12})),
\]
which can be rewritten as
\[
S=A_1 + B_2 + C_{02} + C_{23} - D_{03}.
\]
Let us denote by $\BD^+(X)$ the set containing the boundary divisors in $X$, the divisors $H_{\{01\}\{23\}}$, 
$H_{\{02\}\{13\}}$, $H_{\{03\}\{12\}}$ and the divisor $S$, and by $\cE \subseteq \Eff(X)$ the cone generated by the numerical classes of the divisors in $\BD^+(X)$.

\begin{remark}\label{rem:effgen} 
Knowing the intersection numbers of a basis of $\Nu(X)$ with every divisor in $\BD^+(X)$, we can choose a basis of $\NU(X)$ and compute the whole structure of the cone $\cE$ with {\tt SageMath}, verifying that the class of every divisor in $\BD^+(X)$ generates an extremal ray of $\cE$.
\end{remark}

\begin{theorem}
With the above notation,
\[
\Eff(X) = \cE.
\]
\end{theorem}

\begin{proof}
By \cite[Remark 2.19]{Cas13} 
\begin{equation}\label{eq:eff}
\Eff(X) = \Mov(X) + \R_+[E_1] + \dots + \R_+[E_s],
\end{equation}
where the $E_i$ are the exceptional divisors of the elementary divisorial rational contractions of $X$, i.e., of the elementary divisorial contractions of the SQMs of $X$.

\medskip
\noindent{\bf Step 1:} Every divisor $E_i$ is an element of $\BD^+(X)$.
\par\medskip
Let $\psi: X \dashrightarrow \widehat{X}$ be a small $\Q$-factorial modification of $X$, and let $\varphi_R$ be a divisorial contraction of $\widehat{X}$, associated with a ray $R=\R_+[\Gamma]$.
If a boundary divisor has negative intersection with $\Gamma$, then this divisor is the exceptional locus of $\varphi_R$, so we may assume that all the boundary divisors are non-negative on $\Gamma$.

In view of Lemma \ref{lem:extremal1} we then have $\ell(R) \ge 2$, hence  $\varphi_R$ is the blowup of a smooth point and $\Exc(R) \simeq \PP^2$ (cf. \cite[Theorem~5.1]{AO2}). The small modification $\psi: X \dashrightarrow \widehat{X}$ factors via a finite number of Atiyah flops (see Section \ref{ssec:moricone}). The divisor 
$\Exc(R)$ cannot be disjoint from all the flopped curves, otherwise $\psi$ would be an isomorphism in a neighborhood of $\Exc(R)$ and $X$ would have a Mori contraction which is the blowdown of $\Exc(R)$ to a point.

We can check that the numerical classes of the curves $A_i \tast B_j$ are a set of generators for the linear span of the face of $\cNE{X}$ on which $-K_X$ is trivial,
hence $\Exc(R)$ must meet at least a flopped curve in the class $[-A_i\tast B_j]$.

We assume first that $i=j$, and, without loss of generality, we consider the case $i=j=0$. The irreducible curve in the class $[-A_0 \tast B_0]$ is contained in $\widehat{D}_{01}, \widehat{D}_{02},\widehat{D}_{03}$ (the strict transforms in $\widehat{X}$ of the divisors ${D}_{01}$, ${D}_{02}$, ${D}_{03}$), hence these divisors meet $\Exc(R)$ and therefore their  intersection number with $\Gamma$ is positive. 

From formula (\ref{eq:multican2}), we get that $\widehat{D}_{ij} \cdot \Gamma=1$ and that all the other boundary divisors have intersection number zero with $\Gamma$. We can then check that $\widehat{S} \cdot \Gamma =-1$, hence $\widehat{S} = \Exc(R)$.
We can also compute the numerical class of $\Gamma$, which is
\begin{equation}\label{eq:gcs}
 [A_0\tast B_0 +A_1\tast B_1 +C_{01} \tast C_{23} +A_0\tast D_{01} +B_0\tast D_{01}].
\end{equation}
Note that the divisor $S$ is $\FO$-invariant, hence different choices of $i=j$ will produce the same output. 

We assume now that $i \not=j$, and we consider, without loss of generality, the case $i=0$, $j=1$. The irreducible curve in the class $[-A_0 \tast B_1]$ is contained in $\widehat{D}_{01}, \widehat{C}_{12}, \widehat{C}_{13}$. From formulae (\ref{eq:multican1}) and (\ref{eq:multican2}) we get that
$\widehat{D}_{01}, \widehat{D}_{23}, \widehat{C}_{12}, \widehat{C}_{13},\widehat{C}_{02},\widehat{C}_{03}$ have degree one on $\Gamma$, and all the other boundary divisors have intersection number zero with $\Gamma$. We can then check that $\widehat{H}_{\{01\}\{23\}} \cdot \Gamma =-1$, hence $\widehat{H}_{\{01\}\{23\}} = \Exc(R)$.
We can also compute the numerical class of $\Gamma$, which is
\begin{equation}\label{eq:gcs}
[A_0\tast B_1 +A_3\tast B_2 + A_0\tast C_{12} + B_1\tast C_{12} - C_{03}\tast C_{12}].
\end{equation}
With different choices of the pair $i,j$ we obtain that $\Exc(R)$ is another divisor of type $\widehat H_{\{ij\}\{kl\}}$.
\par
\medskip
\noindent{\bf Step 2:} $\Mov(X) \subseteq \cE$.
\par\medskip
 We will show the dual statement $\cE^\vee \subseteq \Mov(X)^\vee$.
By the weak duality Theorem for Mori Dream Spaces,  \cite[Theorem 4.3]{BDPS}, the cone $\Mov(X)^\vee$ equals the cone $\cC_1^{bir}(X)$, generated by all the curves moving in codimension one in some SQM of $X$.
It will then be enough to construct a cone $\cC$ such that $\cE^\vee \subseteq \cC  \subseteq  \cC_1^{bir}(X)$.
We will consider the cone $\cC$ generated by the classes of curves in the orbits of
\begin{enumerate}
\item $A_0\tast C_{23}$,
\item $A_0\tast D_{01}$,
\item $A_0\tast B_1+A_0\tast D_{01}$, $A_0\tast B_1+A_0\tast C_{12}$, $A_0\tast B_0+A_0\tast D_{01}$,
\item $A_0\tast B_0 +A_1\tast B_1 +C_{01} \tast C_{23} +A_0\tast D_{01} +B_0\tast D_{01},$
\item $A_0\tast B_1 +A_3\tast B_2 + A_0\tast C_{12} + B_1\tast C_{12} - C_{03}\tast C_{12}.$
\end{enumerate}
The deformations of the curves in $(1)$--$(3)$ cover the boundary divisors $C_{23}, D_{01}$ and $A_0$, respectively, while in the last two cases, as seen in Step 1, the curves cover the exceptional divisor of a contraction of a small $\Q$-factorial modification of $X$; we then have an inclusion $\cC \subseteq \cC_1^{bir}(X)$.
To finish the proof, we have checked with {\tt SageMath} that also the inclusion $\cE^\vee \subseteq \cC$ holds.
\end{proof}


\bibliographystyle{plain}
\bibliography{bibliomin}


\end{document}